\documentclass[a4paper]{amsart}
\usepackage[all]{xy}
\usepackage[colorlinks=true]{hyperref}
\usepackage{enumerate}
\usepackage{amssymb}
\newtheorem{thm}{Theorem}[section]
\newtheorem*{thmx}{Theorem}             % not numbered

\newtheorem{prop}[thm]{Proposition}
\newtheorem{lem}[thm]{Lemma}

\theoremstyle{definition}
\newtheorem{defn}[thm]{Definition}
\newtheorem{ex}[thm]{Example}

\theoremstyle{remark}
\newtheorem{rem}[thm]{Remark}

\newcommand{\Rr}{\mathbb R}

\renewcommand{\d}{\mathrm d}                           % Derivation

\newcommand{\br}{\left[\cdot , \cdot \right]}                % Empty Lie bracket
\newcommand{\brr}[1]{\left[#1\right]}                  % Lie bracket
\newcommand{\X}{\ensuremath{\mathfrak{X}}}
\newcommand{\F}{\ensuremath{\mathcal{F}^B}}

\newcommand{\Pa}{\ensuremath{\mathcal{P}}}
\newcommand{\Q}{\ensuremath{\mathcal{Q}}}
\newcommand{\Ha}{\ensuremath{\mathcal{H}}}

\newcommand{\lie}{\mathcal{L}}
\newcommand{\h}{\ensuremath{\mathfrak{h}}}
\newcommand{\g}{\ensuremath{\mathfrak{g}}}
\newcommand{\lxi}{\ensuremath{\overleftarrow{\xi}}}
\newcommand{\C}{\ensuremath{\mathcal{C}}}  % Comcomitant
\newcommand{\T}{{\mathcal{T}}}             % Torsion
\renewcommand{\d}{\mathrm d}               % differential
\newcommand{\smc}{\mbox{\,\tiny{$\circ $}\,}}         %small composition circle

\DeclareMathOperator{\rank}{rank}       % rank of a vector bundle
\DeclareMathOperator{\im}{Im}           % Image
\newcommand{\al}{\alpha}
\newcommand{\be}{\beta}

\begin{document}

\title[Reduction of Poisson-Nijenhuis Lie algebroids]{Reduction of Poisson-Nijenhuis Lie algebroids
 to symplectic-Nijenhuis Lie algebroids with a nondegenerate Nijenhuis
 tensor}

\author{\noindent Antonio De Nicola}
\address{A.\ De Nicola: CMUC, Department of Mathematics, University of Coimbra, Coimbra, Portugal}
\email{antondenicola@gmail.com}

\author{Juan Carlos Marrero}
\address{\noindent J.\ C.\ Marrero: Unidad Asociada ULL-CSIC "Geometr{\'\i}a Diferencial y Mec\'anica Geo\-m\'e\-tri\-ca"
Departamento de Matem\'atica Fundamental, Facultad de
Matem\'aticas, Universidad de la Laguna, La Laguna, Tenerife,
Canary Islands, Spain} \email{jcmarrer@ull.es}

\author{Edith Padr{\'o}n}
\address{\noindent  E.\ Padr{\'o}n: Unidad Asociada ULL-CSIC "Geometr{\'\i}a Diferencial y Mec\'anica Geo\-m\'e\-tri\-ca"
Departamento de Matem\'atica Fundamental, Facultad de
Matem\'aticas, Universidad de la Laguna, La Laguna, Tenerife,
Canary Islands, Spain} \email{mepadron@ull.es}

\thanks{ \noindent\\{\bf Mathematics Subject Classifications:} 53D17
(Primary), 17B62, 17B66, 37J05, 37J35, 37K10, 53D05 (Secundary).\\
{\bf Keywords:} Lie algebroids, Poisson-Nijenhuis Lie algebroids, symplectic-Nijenhuis Lie algebroids, nondegenerate Nijenhuis tensors,
Reduction, bi-Hamiltonian systems.
}

\begin{abstract}
We show how to reduce, under certain regularities conditions,  a
Poisson-Nijenhuis Lie algebroid to a symplectic-Nijenhuis Lie
algebroid with a nondegenerate Nijenhuis tensor. We generalize the
work done by Magri and Morosi for the reduction of Poisson-Nijenhuis
manifolds. The choice of the more general framework of Lie
algebroids is motivated by the geometrical study of some reduced
bi-Hamiltonian systems. An explicit example of reduction of a
Poisson-Nijenhuis Lie algebroid is also provided.
\end{abstract}

%%% ----------------------------------------------------------------------
\maketitle
%%% ----------------------------------------------------------------------

%%%%%%%%%%%%%%%%%%%%%%%%%%%%%%%%%%%%
%%%%%%%%%%%%%%%%%%%%%%%%%%%%%%%%%%%%
\section{Introduction}             %
\label{sec:introduction}           %
%%%%%%%%%%%%%%%%%%%%%%%%%%%%%%%%%%%%
%%%%%%%%%%%%%%%%%%%%%%%%%%%%%%%%%%%%
Poisson-Nijenhuis structures on manifolds were introduced by Magri
and Morosi \cite{MagriMorosi} and then intensively studied by many
authors \cite{KM,Fer,SaVe,Va,VeSaCr}. Recall that a
Poisson-Nijenhuis manifold consists of a triple $(M, \Lambda, N)$,
where $M$ is a manifold endowed with a Poisson bivector field
$\Lambda$ and a $(1,1)$-tensor $N$ whose Nijenhuis torsion vanishes,
together with some compatibility conditions between $\Lambda$ and
$N$. Poisson-Nijenhuis manifolds are very important in the study of
integrable systems since they produce bi-Hamiltonian systems
\cite{KM,Fer,MagriMorosi}. In particular, Magri and Morosi showed
how to reduce a Poisson-Nijenhuis manifold to a nondegenerate one,
i.e., one where the Poisson structure is actually symplectic and the
Nijenhuis tensor is kernel-free. In this paper we show how to
perform the same process of reduction in the more general framework
of Lie algebroids. This type of structures have deserved a lot of interest in relation with
the formulation of the Mechanics  on disparate situations as systems with symmetry, systems evolving on semidirect products, Lagrangian and Hamiltonian systems on Lie algebras, and field theory equations (see, for instance, \cite{CLMM,LMM} and the references therein).

More precisely, in this paper we will see how to reduce a Poisson-Nijenhuis Lie
algebroid to a symplectic-Nijenhuis Lie algebroid with a nondegenerate
Nijenhuis tensor. One could wonder about the interest of such a
generalization. However, we show that working in the framework of
Poisson-Nijenhuis Lie algebroids one may understand the geometrical
structure of some physical examples related with bi-Hamiltonian
systems and hence it is not a mere academic exercise. Indeed we
present, as a motivating example, the study of the classic Toda
lattice which, as is well known, admits a Poisson-Nijenhuis
structure on $\Rr^{2n}$. Nevertheless, when switching to the more
convenient Flaschka coordinates, one sees that the Poisson-Nijenhuis
structure is lost, since there is no more a recursion operator
connecting the hierarchy of Poisson structures. Nevertheless, the
Poisson-Nijenhuis structure can be recovered if the system is
described as a Lie algebroid (see also \cite{Cas}).

The paper is organized as follows. In Section 2 we recall the notion
of Poisson-Nijenhuis Lie algebroid. The introduction of this structure is then motivated by describing the example of the Toda lattice (see  \cite{Cas}).
Moreover, we show how this example can be framed in a more general
case by considering a $G$-invariant Poisson-Nijenhuis structure on
the total space $M$ of a $G$-principal bundle. Such a structure, in
general does not induce a Poisson-Nijenhuis structure on $M/G$.
Nevertheless, it gives rise to a Poisson-Nijenhuis Lie algebroid on
the associated Atiyah bundle, which allows to build the
bi-Hamiltonian system in the reduced space $M/G$. In the following
sections we present the reduction of Poisson-Nijenhuis Lie
algebroids. The reduction process is carried on in two steps. The
first step, described in Section 3, consists in selecting a
generalized foliation $D=\rho_A(P^\sharp A^*)$ on the given
Poisson-Nijenhuis Lie algebroid $(A,\br_A,\rho_A, P,N)$ and then
showing that restricting on each leaf $L$ of $D$ one obtains a
symplectic-Nijenhuis Lie algebroid structure. The leaves of the
foliation $D$ are generally larger than those of the symplectic
foliation of the induced Poisson structure on the base manifold. In
Section 4 we deal with Lie algebroid epimorphisms introducing the
notion of projectability of Poisson-Nijenhuis structures. We prove
that given a projectable Poisson-Nijenhuis structure on a Lie
algebroid and a Lie algebroid epimorphism we obtain a
Poisson-Nijenhuis structure on the target Lie algebroid. Finally, we
introduce the notion of Poisson-Nijenhuis Lie algebroid morphism. In
Section 5 we study the reduction of a Lie algebroid by the foliation
generated by the vertical and complete lifts of the sections of a
Lie subalgebroid using an epimorphism of Lie algebroids.
In Section 6, we use the previous constructions to obtain a reduced
symplectic-Nijenhuis Lie algebroid with a nondegenerate Nijenhuis
tensor from an arbitrary symplectic-Nijenhuis Lie algebroid, under
suitable conditions. In this way we complete the second and final
step of the process of reduction. By putting together the two steps,
we obtain our main result, which is the following one.
\begin{thmx}%\label{thm:symplectic:Nijenhuis:nondegenerate:N}
Let $(A,\br_A,\rho_A, P,N)$ be a Poisson-Nijenhuis Lie algebroid
such that:
\begin{enumerate}
\item[$i)$]
The Poisson structure $P$ has constant rank in the leaves of the
foliation $D=\rho_{A}(P^\sharp(A^\ast))$.
\end{enumerate}
\noindent If $L$ is a leaf of $D$, then, we have a
symplectic-Nijenhuis Lie algebroid structure
$(\br_{A_L},\rho_{A_L},\Omega_L, N_L)$ on
$A_L=P^\sharp(A^\ast)_{|L}\to L$.

Assume, moreover, that
\begin{enumerate}
\item[$ii)$] The induced Nijenhuis tensor $N_L:A_L\to A_L$ has constant Riesz index $k$;
\item[$iii)$] The dimension of the subspace $B_x=\ker N_x^k$ is constant, for all $x\in L$ (thus, $B=\ker N_L^k$ is a vector subbundle of $A$);
\item[$iii)$] The foliations $\rho_A(B)$ and $\mathcal{F}^B$ are regular, where
$$(\mathcal{F}^B)_a=\{X^c(a)+Y^v(a)/X,Y\in \Gamma(B)\},\mbox{ for }a\in A_L$$
\item[$iv)$] \emph{(condition $\mathcal{F}^{B})$}\; For all $x\in L$, $a_x-a'_x\in B_x$ if $a_x$ and $a'_x$ belong to the same leaf of the
foliation $\mathcal{F}^{B}$. %%%\emph{(condition \F)}\;
\end{enumerate}
Then, we obtain a symplectic-Nijenhuis Lie algebroid structure
\linebreak
$(\br_{\widetilde{A_L}},\rho_{\widetilde{A_L}},\widetilde{\Omega_L},\widetilde{N_L})$
on the vector bundle
$\widetilde{A_L}=A_L/\mathcal{F}^{B}\to\widetilde{L}=L/\rho_{A_L}(B)$ with $\widetilde{N_L}$ nondegenerate.
\end{thmx}

The Magri-Morosi reduction (see \cite{MagriMorosi}) of a Poisson-Nijenhuis manifold $M$  is recovered from this result when we consider the tangent bundle $TM$ with its standard Lie algebroid structure.

The last section of the paper contains an explicit example of
reduction of a Poisson-Nijenhuis Lie algebroid which illustrates our
theory. This is obtained by considering a Lie group $G$ which is the
semidirect product of two Lie groups. We construct a $G$-invariant
Poisson-Nijenhuis structure on the cotangent bundle $T^*G$ and then
we obtain a Poisson-Nijenhuis structure on the associated Atiyah Lie
algebroid which is degenerate. Thus, it may be effectively reduced, according to our main
theorem.

%%%%%%%%%%%%%%%%%%%%%%%%%%%%%%%%%%%%
%%%%%%%%%%%%%%%%%%%%%%%%%%%%%%%%%%%%
%%%%%%%%%%%%%%%%%%%%%%%%%%%%%%%%%%%%
\section{Poisson-Nijenhuis Lie algebroids: a motivating example}            %
\label{sec:preliminaries}          %
%%%%%%%%%%%%%%%%%%%%%%%%%%%%%%%%%%%%
%%%%%%%%%%%%%%%%%%%%%%%%%%%%%%%%%%%%
%%%%%%%%%%%%%%%%%%%%%%%%%%%%%%%%%%%%
In this section we will motivate the introduction of the notion of Poisson-Nijenhuis Lie algebroids with a simple example: the Toda lattice. Firstly, we recall some notions and results about Lie algebroids.

%%%%%%%%%%%%%%%%%%%%%%%%%%%%%%%%%%%%%%%%%%%%%%
A \emph{Lie algebroid} is a vector bundle $\tau_A \colon A \to M$
endowed with
\begin{itemize}
\item[(i)] an \emph{anchor map}, i.e. a vector bundle morphism $\rho_A\colon A \to TM$
\item[(ii)] a Lie bracket $\br_A$ on the space of the sections of $A$, $\Gamma(A)$,
such that the \emph{Leibniz rule},
\[
\brr{X,fY}_A=f \brr{X,Y}_A + \rho_A(X)(f) Y,
\]
is satisfied for all $X,Y\in \Gamma(A)$ and $f\in C^\infty(M).$
\end{itemize}
We denote such a Lie algebroid by $(A, \br_A,\rho_A)$ or simply by
$A$. In such a case  the
map $\rho_A$ induces a morphism of Lie algebras from
$(\Gamma(A),\br_A)$ to $({\mathfrak X}(M),\br)$ which we denote by the
same symbol. For further details about Lie algebroids see e.g.
\cite{Ma}.

Now, we describe some examples of Lie algebroids which will useful for our purposes.

%%%%%%%%%%%%%%%%%%example
\begin{ex}{\bf The standard Lie algebroid structure on the tangent bundle of a manifold.} {\rm The tangent  bundle $TM$ of a manifold $M$ is a Lie algebroid. The Lie bracket on the space of sections of  $TM\to M$ is the standard Lie bracket of vector fields  and the anchor map is the identity map.

\hfill{$\Diamond$}}
\end{ex}
\begin{ex}{\bf The Atiyah algebroid associated with a principal
$G$-bundle.}\label{Atiyah} {\rm Let $p:M \to M/G$ be a principal
$G$-bundle. It is well-known that the tangent lift of the principal
action of $G$ on $M$ induces a principal action of $G$ on $TM$ and
the space of orbits $TM/G$ of this action is a vector bundle over
$M/G$ with vector bundle projection $\tau_{TM/G}: TM/G \to M/G$
given by
\[
\tau_{TM/G}([v_{x}]) = p(x), \makebox[.6cm]{} \forall v_x \in T_x M.
\]
Furthermore, the space of sections $\Gamma({TM/G})$ may be
identified with the set of $G$-invariant vector fields on $M$ and
the Lie bracket of two $G$-invariant vector fields on $M$ is still
$G$-invariant. Thus, the standard Lie bracket of vector fields induces a Lie bracket
$\br_{TM/G}$ on the space
$\Gamma({TM/G})$ in a natural way.

On the other hand, the anchor map $\rho_{TM/G}: (TM)/G \to T(M/G)$ is
given by
\[
\rho_{TM/G}([v_{x}]) = (T_x p)(v_x), \mbox{ for } v_{x} \in T_xM,
\]
where $Tp: TM \to T(M/G)$ is the tangent map to the principal bundle
projection $p: M \to M/G$.

The resultant Lie algebroid $((TM)/G, \br_{TM/G}, \rho_{TM/G})$ is
called {\em the Atiyah algebroid} associated with the principal
$G$-bundle $p:M \to M/G$.
\hfill{$\Diamond$}
 }
\end{ex}
%%%%%%%%%%%%%%%%
If  $(A, \br_A,\rho_A)$ is a Lie algebroid on the manifold $M$, we denote by $$d^A:\Gamma(\wedge ^\bullet A^*)\to \Gamma(\wedge^{\bullet +1} A^*)$$ the
\emph{Lie algebroid differential}  (see \cite{Ma}). Moreover, if $X$ is a
section of $A$, one may introduce, in a natural way, the \emph{Lie
derivative with respect to $X$} as the operator $\lie^A_X\colon
\Gamma(\wedge^{\bullet} A^\ast) \to \Gamma(\wedge^{\bullet} A^\ast)$ given by
\begin{equation}\label{eq:Cartan:formula}
\lie^A_X=i_X\smc\d^A + \d^A\smc i_X.
\end{equation}
It is  easy to prove  that the Lie derivative $\lie^A_X$ and the Lie
bracket $\br_A$ are related by
\begin{equation}\label{wellknown}
\lie^{A}_X i_Y=i_Y\lie^{A}_X + i_{\brr{X,Y}_A},\mbox{ with $X,Y\in \Gamma(A)$.}
\end{equation}

The Lie algebra bracket $\br_A$ on $\Gamma(A)$ can be extended to
the exterior algebra $(\Gamma(\wedge^{\bullet} A).$ The resulting bracket is called {\it Schouten bracket } (see e.g. \cite{Ma}). So, in the case of the bracket of a section $X$ of $A\to M$ and a section $R$ on $\wedge^rA\to M,$  we have that
 \begin{equation}\label{9'}
[X,R]_A(\alpha_1,\dots ,\alpha_r)=\rho_A(X)(R(\alpha_1,\dots ,\alpha_r))-\sum_{i=1}^rR(\alpha_1,\dots ,\lie_X^A\alpha_i,\dots ,\alpha_r)
\end{equation}
 for  $\alpha_1,\dots ,\alpha_r\in \Gamma(A^*).$

Now, let
$P$ be a section of the vector bundle $\wedge^2 A\to M$. We denote by
$P^\sharp$ the usual bundle map
\begin{equation}\label{eq:P:sharp}
P^\sharp \colon   A^\ast  \longrightarrow A,\;\;\;
         \al \longmapsto   P^\sharp(\al)=i_\al P.
\end{equation}
We say that $P$ defines a \emph{Poisson structure on $A$} if
$[P,P]_A=0$. In this case, the  bracket on the sections of
$A^\ast$ defined by
\begin{align}\label{eq:Koszul:bracket}
[\al,\be]_{P}&=\lie^A_{P^\sharp\al}\be-\lie^A_{P^\sharp\be}\al-\d^A\left(P(\al,\be)\right),
\quad \al,\be\in\Gamma(A^\ast),
\end{align}
is a Lie bracket, $P^\sharp:(\Gamma(A^*), [\cdot,\cdot]_P)\to (\Gamma(A),[\cdot,\cdot]_A)$ is a Lie algebra morphism
and the triple $A^\ast_P=(A^\ast, \br_{P},
\rho_{A}\smc P^\sharp)$ is a Lie algebroid \cite{M-X}. In fact, the
pair $(A,A^\ast_P)$ is a special kind of a Lie bialgebroid  called a
\emph{triangular Lie bialgebroid} \cite{M-X}. A Poisson structure
$P\in \Gamma(\wedge^2 A)$ on a Lie algebroid $(A,\br_A,\rho_A)$
induces a Poisson structure $\Lambda\in \Gamma(\wedge^2 TM)$ on the
base manifold $M$, defined by
\begin{equation}\label{eq:Poisson:induced}
\Lambda^\sharp = \rho_A \smc P^\sharp \smc\rho_A^*.
\end{equation}
A \emph{symplectic structure} on the Lie algebroid
$(A,\br_A,\rho_A)$ is a section $\Omega_A$ of the vector bundle
$\wedge^2 A^\ast \to M$ such that $d^A\Omega_A=0$ and $\Omega_A$ is
nondegenerate. In such a case, the map
$\Omega_A^{\flat}:\Gamma(A)\to\Gamma(A^\ast)$ given by
\[
\Omega_A^{\flat}(X)=i_X\Omega_A, \qquad \mbox{for }X\in\Gamma(A),
\]
is an isomorphism of $C^\infty(M)$-modules. Thus, one can define
from $\Omega_A$ a Poisson structure
\begin{equation}\label{eq:Poisson:symplectic}
P_{\Omega_A}(\al,\be)=\Omega_A((\Omega_A^{\flat})^{-1}(\al),(\Omega_A^{\flat})^{-1}(\be)),
\qquad \mbox{for }\al,\be\in\Gamma(A^\ast).
\end{equation}

Let $(A,\br_A,\rho_A)$ be a Lie algebroid over a manifold $M$. The
torsion  of a bundle map $N:A\to A$ (over the identity) is defined
by
\begin{equation}\label{eq:Nijenhuis}
\T_N(X,Y):=[NX,NY]_A-N[X,Y]_N, \quad  X,Y\in \Gamma(A),
\end{equation}
where $\br_N$ is given by
\begin{equation}\label{eq:deformed:bracket}
[X,Y]_N:=[NX,Y]_A+[X,NY]_A-N[X,Y]_A,\quad X,Y\in \Gamma(A).
\end{equation}

When $\T_N=0$, the bundle map $N$ is called a \emph{Nijenhuis
operator}, the triple $A_N=(A,\br_N,\rho_N=\rho\smc N)$ is a new
Lie algebroid  and  $N:A_N\to A$ is a Lie algebroid morphism (see
\cite{GU1,KM}).

Now, if $P\in\Gamma(\wedge^2 A)$ is a Poisson structure on $A$, we
say that a bundle map $N:A\to A$ is \emph{compatible} with $P$ if
$N\smc P^\sharp=P^\sharp\smc  N^*$ and the \emph{Magri-Morosi concomitant}
\begin{equation}\label{15'}
\C(P,N)(\al,\be)= \brr{\al,\be}_{NP}-\brr{\al,\be}_P^{N^\ast},\;\;\; \mbox{ for }\alpha,\beta\in \Gamma(A^*)
\end{equation}
vanishes, where $\br_{NP}$ is the  bracket defined by the section
$NP\in\Gamma(\wedge^{2} A)$ in a similar way as in
\eqref{eq:Koszul:bracket}, and $\br_P^{N^\ast}$ is the Lie bracket
obtained from the Lie bracket $\br_P$ by deformation along the
dual map $N^\ast:A^\ast\to A^\ast$, i.e.,
\begin{equation}\label{15''}
\brr{\al,\be}_P^{N^\ast}=\brr{N^\ast\al,\be}_P+\brr{\al,N^\ast\be}_P-N^\ast\brr{\al,\be}_P.
\end{equation}

\begin{defn}(\cite{GU1,KM})\label{def:PN:Lie:alg}
A \emph{Poisson-Nijenhuis Lie algebroid} $(A, P, N)$ is a Lie
algebroid $A$ equipped with a Poisson structure $P$ and a
Nijenhuis operator $N:A\to A$ compatible with $P$.
\end{defn}

A Poisson-Nijenhuis structure on a manifold $M$ \cite{MagriMorosi} may be seen as a Poisson-Nijenhuis Lie algebroid structure on the tangent  bundle $TM$ with its standard Lie algebroid structure.

If, in particular, the Poisson tensor $P$ in Definition
\ref{def:PN:Lie:alg} is nondegenerate, i.e. it comes from a
symplectic structure $\Omega_A$ on $A$ like in
\eqref{eq:Poisson:symplectic}, then $(A, \Omega_A, N)$ is said to
be a \emph{ symplectic-Nijenhuis Lie algebroid.} This is the case of two compatible Poisson $2$-sections $P_0$ and $P_1$,  where $P_0$ is associated with a symplectic structure.

%The well-known notion of Poisson-Nijenhuis manifold is obtained as a
%special case by considering a Poisson-Nijenhuis Lie algebroid
%structure $(A, \Lambda, N)$ on the canonical Lie algebroid $A=(TM,
%\br, \rho=id_{TM})$ defined on the tangent bundle of a manifold $M$.
%In this case, $\Lambda\in\Gamma(\wedge^2 TM)$ is a Poisson structure
%in the classical sense and the compatible Nijenhuis operator
%$N:TM\to TM$ is a Nijenhuis tensor on the manifold $M$. Then,
%$(M,\Lambda,N)$ is said to be a \emph{Poisson-Nijenhuis manifold}.
%An example of Poisson-Nijenhuis manifold is given by a
%\emph{bi-Hamiltonian manifold}, i.e. a manifold $M$ endowed with two
%Poisson structures $\Lambda_1,\Lambda_2$ compatible with each other
%(i.e. $[\Lambda_1,\Lambda_2]=0$) such that the first one is
%nondegenerate. Thus, $(M,\Lambda_1,N)$ is a Poisson-Nijenhuis
%manifold where $$N:=\Lambda_2^\sharp \smc (\Lambda_1^\sharp)^{-1}.$$

\begin{ex}\label{ex:TL}{\it \bf The Toda lattice.} {\rm The finite, non-periodic Toda lattice (see, for instance, \cite{Cas,Fer,NM}) is a
system of $n$ particles on the line under exponential interaction
with nearby particles. Its phase space is $\Rr^{2n}$ with canonical
coordinates $(q^i,p_i)$ where $q^i$ is the displacement of the
$i$-th particle from its equilibrium position and $p_i$ is the
corresponding momentum. This system is particularly interesting when we consider exponential forces. Then,  the Hamiltonian function associated with the equations of motion is
$$H_1=\frac{1}{2}\sum_{i=1}^n p_i^2 + \sum_{i=1}^{n-1} e^{(q^i-q^{i+1})}.$$
Now, we consider the following two compatible Poisson structures on $\Rr^{2n}$
$$\Lambda_0=\sum_{i=1}^n \frac{\partial }{\partial q^i}\wedge \frac{\partial }{\partial p_i},$$
$$\Lambda_1=-\sum_{i<j} \frac{\partial }{\partial q^i}\wedge \frac{\partial }{\partial q^j}+
\sum_{i=1}^n p_i\frac{\partial }{\partial q^i}\wedge \frac{\partial
}{\partial p_i} + \sum_{i=1}^{n-1} e^{(q^i-q^{i+1})}\frac{\partial
}{\partial p_{i+1}}\wedge \frac{\partial }{\partial p_i}$$ which determine the Poisson-Nijenhuis structure $(N={\Lambda}_1^\sharp\circ ({\Lambda}_0^\sharp)^{-1},\Lambda_0),$ whose  Nijenhuis tensor $N$ is characterized by
\begin{equation}\label{N}\begin{array}{rcl}
N(\displaystyle\frac{\partial }{\partial q^i})&=&p_i\displaystyle\frac{\partial }{\partial q^i}-e^{q^{i-1}-q^i}\displaystyle\displaystyle\frac{\partial}{\partial p_{i-1}}+ e^{q^{i}-q^{i+1}}\displaystyle\frac{\partial}{\partial p_{i+1}}\\[8pt]
N(\displaystyle\frac{\partial }{\partial p_i})&=& p_i\displaystyle\frac{\partial }{\partial p_i} + \displaystyle\sum_{j<i}\displaystyle\frac{\partial }{\partial q^j}-\displaystyle\sum_{j>i}\displaystyle\frac{\partial }{\partial q^j}.
\end{array}\end{equation}

Note
that $\Lambda_0$ is the Poisson bivector corresponding to the
canonical symplectic structure of $\Rr^{2n}$. Furthermore, the Hamiltonian vector field $\Ha_{H_1}^{\Lambda_0}$ is bi-Hamiltonian. In fact,
$$\Ha_{H_1}^{\Lambda_0}=\Lambda_0^\sharp(dH_1)=\Lambda_1^\sharp(dH_0),$$
with $H_0=\sum_{i=1}^n p_i$.

In what follows, we will reduce the bi-Hamiltonian structure of the Toda lattice using
the action of $\Rr$ over $\Rr^{2n}$  given by
\begin{align*}
{\Rr} \times {\Rr}^{2n}&\longrightarrow {\Rr}^{2n}\\
 (t,(q^i,p_i))&\longmapsto (q^i+t,p_i)
\end{align*}
which induces the principal bundle
$$\pi: {\Rr}^{2n}\longrightarrow {\Rr}^{2n}/\Rr.$$
Note that ${\Rr}^{2n}/\Rr$ may be identified with
$(\Rr^{+})^{n-1}\times \Rr^n$ by
\begin{equation}\label{iden}
{\Rr}^{2n}/\Rr \longrightarrow (\Rr^{+})^{n-1}\times \Rr^n,\;\;\;
 (q^i,p_i)  \longmapsto (e^{(q^i-q^{i+1})},p_i).
\end{equation}
This identification corresponds to the choice of the so called
{\it Flaschka coordinates } which are actually global coordinates on
${\Rr}^{2n}/\Rr$, usually denoted by
$(a_1,\ldots,a_{n-1},$ $b_1,$ $\ldots,b_n)$.
 The Poisson structures
$\Lambda_0$ and $\Lambda_1$ are $\Rr$-invariant so that they descend
to the quotient ${\Rr}^{2n}/\Rr \cong (\Rr^{+})^{n-1}\times \Rr^n$.
The reduced Poisson structures are
\begin{equation}\label{eq:lambda:bar:0}
\begin{array}{rcl}
\bar\Lambda_0&=&\displaystyle\sum_{i=1}^{n-1} a_i\displaystyle\frac{\partial }{\partial
a_i}\wedge \left(\displaystyle\frac{\partial }{\partial b_i} - \displaystyle\frac{\partial
}{\partial b_{i+1}}\right),\\[10pt]
\bar\Lambda_1&=&\displaystyle\sum_{i=1}^{n-1} a_i\displaystyle\frac{\partial }{\partial
a_i}\wedge \left(b_i\frac{\partial }{\partial b_i}
 - b_{i+1}\displaystyle\frac{\partial }{\partial b_{i+1}} \right)\\
& &+\displaystyle\sum_{i=1}^{n-1} a_i\displaystyle\frac{\partial }{\partial b_{i+1}} \wedge
\frac{\partial }{\partial b_i} + \displaystyle\sum_{i=1}^{n-2} a_i
a_{i+1}\displaystyle\frac{\partial }{\partial a_{i+1}}\wedge\displaystyle\frac{\partial
}{\partial  a_i}.
\end{array}
\end{equation}
 These bivectors are again compatible and moreover we obtain by projection a hierarchy of compatible Poisson structures on the reduced space. However, they cannot be related
 through a recursion tensor $\bar N$. Indeed, if this were the case,
 then
 $$\bar\Lambda_1^\sharp= \bar N \smc \bar\Lambda_0^\sharp.$$
 Thus, using that $\bar\Lambda_0^\sharp(\displaystyle\sum_{i=1}^{n}db_{i})=0,$ we deduce that $\bar\Lambda_1^\sharp(\displaystyle\sum_{i=1}^ndb_{i})=0$ which is not true.

 The problem is that if we want to induce a tensor $\bar N:
T({\Rr}^{2n}/\Rr)\to T({\Rr}^{2n}/\Rr)$ it is necessary that the tensor $N$ on $\Rr^{2n}$ sends
vertical vectors with respect to $\pi: {\Rr}^{2n}\rightarrow
{\Rr}^{2n}/\Rr$ into vertical vectors. Note that $\Lambda_0$ and $N$
are $\Rr$-invariant. However, using   $\ker T\pi=<\displaystyle\sum_{i}\frac{\partial }{\partial q^i}>$  and (\ref{N}),  one deduces that $N(\ker T\pi)\nsubseteq \ker T\pi$ .

Furthermore, the Hamiltonian vector field $\Ha_{H_1}^{\Lambda_0}$ projects just in $\Ha_{\bar{H}_1}^{\bar{\Lambda}_0}$ and
$${\bar{\Lambda}_0}^\sharp d\bar{H}_1=\Ha_{\bar{H}_1}^{\bar{\Lambda}_0}={\bar{\Lambda}^\sharp_1}d\bar{H}_0$$
 with $\bar{H}_1=\displaystyle\frac{1}{2} \displaystyle\sum_{i=1}^{n}b_i^2 + \displaystyle\sum_{i=1}^{n-1}a_i$ and $\bar{H}_0=\displaystyle\sum_{i=1}^n b_i.$

These facts suggest that,  although the structure of Poisson-Nijenhuis can not be reduced, perhaps there exists another structure in a different space from which we may induce the above structures on the reduced space
${\Rr}^{2n}/\Rr.$ The answer to this question is associated with the notion of a Poisson-Nijenhuis Lie algebroid.}

We will describe now the Poisson-Nijenhuis Lie algebroid associated to the reduction of the Toda lattice. Consider the Atiyah algebroid $\tau_A: A=(T\Rr^{2n})/\Rr\to \Rr^{2n}/\Rr$ associated with the principal bundle $\pi:\Rr^{2n}\to \Rr^{2n}/\Rr.$

A global basis of $\Rr$-invariant vector fields on $\Rr^{2n}$ is $$\{e_i=e^{(q^{i+1}-q^i)}\sum_{k=1}^i\frac{\partial }{\partial q^k},\;\;\; e_n=\sum_{k=1}^n\frac{\partial }{\partial q^k}, \;\;\; f_j=\frac{\partial }{\partial p_j}\}_{\kern-3pt\tiny\begin{array}{l}i=1,..,n-1\\ j=1,..,n\end{array}}$$
Note that
$$[e_i,e_j]=[f_i,f_j]=[e_i,f_j]=0$$
for $i,j\in\{1,\dots, n\}$. Moreover, the vector field $e_k$, with
$k\in \{1,\dots, n-1\}$ (respectively, $f_l,$ with $l\in \{1,\dots
,n\}$) is $\pi$-projectable over the vector field
$\displaystyle\frac{\partial }{\partial a^k}$ (respectively,
$\displaystyle\frac{\partial }{\partial b_l}$) on
$(\Rr^+)^{n-1}\times \Rr^n.$ In addition, the vertical bundle of
$\pi$ is generated by the vector field $e_n$.

Thus, the Lie algebroid structure $([\cdot,\cdot]_A,\rho_A)$ on $A$ is characterized by the following conditions

\[
\brr{e_i,e_j}_A=\brr{f_i,f_j}_A=\brr{e_i,f_j}_A=0,
\]
and
\[
\rho_A(e_i)=\frac{\partial }{\partial a_i} \quad (i=1,\ldots, n-1),
\qquad  \rho_A(e_n)=0,\qquad
\rho_A(f_j)=\frac{\partial }{\partial b_j}\quad (j=1,\ldots, n).
\]
We may define the following two Poisson structures on $A$
\[
\begin{array}{rcl}
\pi_0&=&\displaystyle\sum_{i=1}^{n-1}  a_i e_i\wedge (f_i-f_{i+1}) + e_n\wedge f_n\\[10pt]
 \pi_1&=&-\displaystyle\sum_{i=1}^{n-2}   a_i a_{i+1}e_i\wedge e_{i+1} -
a_{n-1}e_{n-1}\wedge e_n + \displaystyle\sum_{i=1}^{n-1} a_i e_i \wedge
(b_if_i-b_{i+1}f_{i+1})\\
&&+ b_n e_n\wedge f_n -\displaystyle\sum_{i=1}^{n-1} a_i f_i\wedge f_{i+1}.
\end{array}\]
These Poisson structures cover ordinary Poisson tensors on the base
manifold $\Rr^{2n}/\Rr$ which are just the Poisson structures $\bar\Lambda_0$ and
$\bar\Lambda_1$ given by \eqref{eq:lambda:bar:0}. Since $\pi_0$ is symplectic, the Poisson structures on
$A$ are related by the recursion operator $N= \pi_1^\sharp\smc
(\pi_0^\sharp)^{-1}$ and $(A,\pi_0,N)$ is a symplectic-Nijenhuis Lie
algebroid.
\hfill{$\Diamond$}
\end{ex}

This example may be framed within a
more general framework as follows.

Let $p:M\rightarrow
\bar{M}=M/G$ be a  principal $G$-bundle. If a $G$-invariant
Poisson-Nijenhuis structure $(\Lambda,N)$ is given on $M$, then in
general we cannot induce  a Poisson-Nijenhuis structure on $M/G$
since the condition $N(\ker Tp)\nsubseteq \ker Tp$ might not be
satisfied. Nevertheless,  we  obtain a reduced Poisson-Nijenhuis Lie
algebroid. In fact, as we know, the space of sections of $\tilde{p}:TM/G\to \bar{M}=M/G$ (respectively, $\tilde{p}^*: (TM/G)^*\cong T^*M/G\to \bar{M}=M/G$) may be identified with the set of $G$-invariant vector fields ${\mathfrak X}^G(M)$ (respectively, $G$-invariant $1$-forms $\Omega^1(M)^G$) on $M$.

Now, since $\Lambda$ and $N$ are $G$-invariant, we deduce that
$$\Lambda(\alpha,\beta) \mbox{ is a $p$-basic function,  for $\alpha,\beta\in \Omega^1(M)^G$}$$
and
$$NX\in {\mathfrak X}^G(M), \mbox{ for }X\in {\mathfrak X}^G(M).$$
Thus, $\Lambda$ (respectively, $N$) induces a section $\tilde\Lambda$ (respectively, $\widetilde{N}$) on the vector bundle $\wedge^2(TM/G)\to \bar{M}=M/G$ (respectively, $TM/G\otimes T^*M/G\to \bar{M}=M/G$) in such a way that
\[
\begin{array}{lcl}\widetilde{\Lambda}(\alpha,\beta)\smc p=\Lambda(\alpha,\beta)&\mbox{ for } &\alpha,\beta\in \Omega^1(M)^G,\\
\tilde{N}X=NX,&\mbox{ for } &X\in {\mathfrak X}^G(M).\end{array}
\]
Moreover,
using the definition of the Lie algebroid structure on the Atiyah algebroid $p:TM/G\to \bar{M}=M/G$ and the fact that $(\Lambda,N)$ is a Poisson-Nijenhuis structure on $M$, we may prove the following result

\begin{prop}\label{2.6}
Let $p:M\to \bar{M}=M/G$ be a principal $G$-bundle and $(\Lambda,N)$ be a $G$-invariant Poisson-Nijenhuis structure on $M$. Then:
\begin{enumerate}
\item[$i)$] $(\Lambda,N)$ induces a Poisson-Nijenhuis Lie algebroid structure $(\tilde\Lambda, \tilde{N})$ on the Atiyah algebroid $\tilde{p}: TM/G\to \bar{M}=M/G$
\item[$ii)$] The Poisson structures $\Lambda$ and $N\Lambda$ on $M$ are $p$-projectable to two compatible Poisson structures $\bar\Lambda$ and $\overline{N\Lambda}$ on $\bar{M}=M/G.$
\item[$iii)$] The Poisson structures on $\bar{M}=M/G$ which are induced by the Poisson bi-sections $\tilde{\Lambda}$ and $\tilde{N}\tilde{\Lambda}$ on the Atiyah algebroid $p:M\to \bar{M}=M/G$ are just $\bar{\Lambda}$ and $\overline{N\Lambda}$, respectively.
\end{enumerate}
\end{prop}

%%%%%%%%%%%%%%%%%%%%%%%%%%%%%%%%%%%%%%%%%%%%%%%%%%%%%%%%%%%%%%%%%%%%%%%%%%%%%%%%%%%%%%%%%%%%%%%%%%%%%%%%%%%%%%%%%%%%%%
\section{Reduction of Poisson-Nijenhuis Lie algebroids by restriction}     %
%%%%%%%%%%%%%%%%%%%%%%%%%%%%%%%%%%%%%%%%%%%%%%%%%%%%%%%%%%%%%%%%%%%%%%%%%%%%
%%%%%%%%%%%%%%%%%%%%%%%%%%%%%%%%%%%%%%%%%%%%%%%%%%%%%%%%%%%%%%%%%%%%%%%%%%%%
We consider the Poisson-Nijenhuis Lie algebroid $A=(T{\Rr}^{2n})/\Rr$ associated with the Toda lattice.
 It is easy to prove that if we restrict to a suitable open subset of the base manifold $\Rr^{2n}/\Rr$
 then $A=(T\Rr^{2n})/\Rr$ is a symplectic-Nijenhuis Lie algebroid with a nondegenerate Nijenhuis tensor.
 The main result of this paper is prove that, under regularities conditions, every Poisson-Nijenhuis algebroid may be reduced to a nondegenerate symplectic-Nijenhuis Lie algebroid. This reduction has two steps. In the first step we obtain a symplectic-Nijenhuis Lie algebroid, and then we will reduce it to a symplectic-Nijenhuis Lie algebroid with a nondegenerate Nijenhuis tensor using a general theory about the projectability of a Poisson-Nijenhuis structure with respect to a Poisson-Nijenhuis Lie algebroid epimorphism.   In this section we will describe the first step which is a  reduction by restriction. Previously,  we recall some notions about Lie algebroid morphisms which will be useful in the sequel.

\subsection{Lie algebroid morphisms and subalgebroids}
Let $\tau_A \colon A \to M$ and $\tau_{\widetilde{A}} \colon
\widetilde{A} \to \widetilde{M}$ be vector bundles. Suppose that
we have a morphism of vector bundles $(F,f)$ from $A$ to
$\widetilde{A}$:

\[%%%COMMUTATIVE DIAGRAM
\xymatrix@C=3pc@R=3pc{A\ar[r]^{F}\ar[d]^{\tau_A} & \widetilde{A}\ar[d]^{\tau_{\widetilde{A}}}\\
            M\ar[r]^{f} & \widetilde{M}}
\]
\vspace{4mm}

\noindent A section of $A$, $X:M\to A,$ is said to be
$F$-\emph{projectable} if there is
$\widetilde{X}\in\Gamma(\widetilde{A})$ such that the following
diagram is commutative:

\[%%%COMMUTATIVE DIAGRAM
\xymatrix@C=3pc@R=3pc{A\ar[r]^{F} & \widetilde{A}\\
            M\ar[u]^{X}\ar[r]^{f} & \widetilde{M}\ar[u]^{\widetilde{X}}}
\]
\vspace{5mm}

\noindent A section $\al \colon M \to \wedge^k A^\ast$ of
$\tau^k_{A^\ast} \colon \wedge^k A^\ast \to M$  is said to be
$F$-\emph{projectable} if there is
$\widetilde{\al}\in\Gamma(\wedge^k \widetilde{A}^\ast)$ such that
$\al=F^\ast\widetilde{\al}$, where $F^\ast
\widetilde{\al}\in\Gamma(\wedge^k A^\ast)$ is defined by
\begin{equation}\label{eq:F*}
(F^\ast
\widetilde{\al})(x)(a_1,\dots,a_k)=\widetilde{\al}({f(x)})(F(a_1),\dots,F(a_k))
\end{equation}
with $x\in M$ and $a_1, \dots, a_k \in A_x$.

Now, we consider Lie
algebroid structures $(\br_A,\rho_A)$ and
$(\br_{\widetilde{A}},\rho_{\widetilde{A}})$ on $A$ and
$\widetilde{A}$, respectively. We say that $(F,f)$ is a \emph{Lie
algebroid morphism} if
\begin{equation}\label{eq:Lie:alg:morphism}
\d^{A} (F^\ast \widetilde{\al})= F^\ast (\d^{\widetilde{A}}
\widetilde{\al}) \quad \mbox{for all }
\widetilde{\al}\in\Gamma(\wedge^k \widetilde{A}^\ast) \mbox{ and
all }k.
\end{equation}

Any Lie algebroid morphism preserves the anchor, i.e.,
\begin{equation}\label{eq:morphism:anchors}
\rho_{\widetilde{A}}\smc F=Tf\smc \rho_A.
\end{equation}
Moreover, if $X$ and $Y$ are $F$-projectable sections on $\widetilde{X}$ and $\widetilde{Y},$ respectively, it follows that $[X,Y]_A$ is a $F$-projectable section on $[\widetilde{X},\widetilde{Y}]_{\widetilde{A}}.$

In addition, if $X\in\Gamma(A)$ is $F$-projectable and
$\widetilde{\al}\in\Gamma(\widetilde{A}^\ast)$, then $\lie^{A}_X
(F^\ast \widetilde{\al})$ is a $F$-projectable section of $A^\ast$.
In fact, using (\ref{eq:Cartan:formula}) and (\ref{eq:Lie:alg:morphism}), we have that
\begin{equation}\label{eq:Lie:reduced}
\lie^{A}_X (F^\ast \widetilde{\al})=
F^\ast(\lie^{\widetilde{A}}_{\widetilde{X}} \widetilde{\al}),
\end{equation}
where $\widetilde{X}\in\Gamma(\widetilde{A})$ satisfies $F\smc
X=\widetilde{X} \smc f$.

Note that if $M=\widetilde{M}$ and $f$ is the identity map for
$M$, then $F:A\to \widetilde{A}$ is a Lie algebroid morphism if
and only if
\begin{equation}\label{eq:Lie:alg:morphism:same:base}
F\brr{X,Y}_A= \brr{FX,FY}_{\widetilde{A}},\qquad
\rho_{\widetilde{A}}(FX)=\rho_A(X)
\end{equation}
for $X,Y\in \Gamma(A).$

A \emph{Lie subalgebroid} is a morphism of Lie algebroids $I\colon
B\to A $ over $\iota\colon N\to M$ such that $\iota$ is an injective immersion and $I_{|B_x}:B_x\to A_{\iota(x)}$ is a
monomorphism, for all $x\in N$ (see \cite{HM}).

\subsection{The first step of the reduction: Reduction of Poisson-Nijenhuis Lie algebroids by restriction}

%%%%%%%%%%%%%%%%%%%%%%%%%%%%%%%%%%%%%%%%%

Let $(A,P)$ be a Poisson Lie algebroid. In order to reduce $A$ to
a symplectic Lie algebroid, let us consider the generalized
distribution $D\subset TM$ defined as follows: for each $x\in M$,
\[
D(x):=\rho_{A}(P^\sharp(A^\ast_x))\subset T_x M.
\]
Since $P^\sharp$ and $\rho_{A}$ are Lie algebroid morphisms over
the identity $id_M:M\to M$, we have
\[
\brr{\rho_{A}(P^\sharp\al),\rho_{A}(P^\sharp\be)}=\rho_{A}(P^\sharp\brr{\al,\be}_{P}),
\]
for any $\al,\be\in \Gamma(A^\ast)$, i.e. $D$ is involutive.
Furthermore, $D$ is locally finitely generated as a $C^\infty(M)$-module. As a consequence $D$
defines a generalized foliation of $M$ in the sense of Sussmann
\cite{Su}. Note that, due to \eqref{eq:Poisson:induced}, the tangent
distribution $S=\Lambda^\sharp(T^*M)$ of the symplectic foliation of
the induced Poisson structure $\Lambda\in\Gamma(\wedge^2TM)$ on the
base manifold $M$ is a subset of
$D=\rho_{A}(P^\sharp(A^\ast))$.

Let $L\subset M$ be a leaf of the foliation $D$ and consider the
subset $A_L:=P^\sharp(A^\ast)_{|L}\subset A$. We assume that the
Poisson structure $P^\sharp: A^\ast \to A$ has constant rank on each
leaf $L$. Then,  $A_L\to L$ is a vector subbundle of the vector bundle $A\to M$ and, since that $\rho_A(A_L)\subseteq TL$, we deduce that the Lie algebroid structure $([\cdot,\cdot]_A, \rho_A)$ on $A$ induces a Lie algebroid structure $([\cdot,\cdot]_{A_L},\rho_{A_L})$ on $A_L.$  In fact, $\rho_{A_L}=(\rho_A)_{|A_L}$ and the Lie bracket $[\cdot,\cdot]_{A_L}$ is characterized by the condition
$$
\brr{P^\sharp\al_{|L},P^\sharp\be_{|L}}_{A_L}=(\brr{P^\sharp\al,P^\sharp\be}_{A})_{|L}=(P^\sharp\brr{\al,\be}_{P})_{|L}
$$
for all $\alpha,\beta\in \Gamma(A^*).$ Note that if $\alpha,\alpha'\in \Gamma(A^*)$ and $P^\sharp(\alpha)_{|L}=P^\sharp(\alpha')_{|L}$ then, using that the restriction to $L$ of $\rho_A(P^\sharp(\beta))$ is tangent to $L$, we obtain that
$$(\brr{P^\sharp\al,P^\sharp\be}_{A})_{|L}=(\brr{P^\sharp\al',P^\sharp\be}_{A})_{|L}.$$

Furthermore, if we denote by $I:A_L\to A$ and $\iota:L\to M$,
respectively, the inclusion mappings of $A_L$ in $A$ and of $L$ in
$M$, then $I$ is a monomorphism of Lie algebroids from $A_L$ to
$A$ over $\iota:L\to M$ so that $A_L$ is a Lie subalgebroid of
$A$.

Now, we will prove that the Lie algebroid $A_L$ is symplectic.

Note that for any $X_L\in\Gamma(A_L)$ there exists a section
$\al\in\Gamma(A^\ast)$ such that $X_L$ $I$-projects on
$P^\sharp\al$, i.e., $I\smc X_L= P^\sharp\al \smc \iota$.

 Let us define a section
$\Omega_L\colon L\to \wedge^2 A_L^\ast$ by setting
\begin{equation}\label{eq:omega:L}
\Omega_L(X_L, Y_L)=P(\al,\be)\smc \iota, \qquad \mbox{for any }
X_L,Y_L \in \Gamma(A_L)
\end{equation}
$\al,\be$ being sections of $A^\ast$ such
that $X_L$ and $Y_L$ $I$-project on $P^\sharp\al$ and
$P^\sharp\be$, respectively. Clearly, $\Omega_L$ is well defined.
Indeed, if $P^\sharp\al\smc \iota=P^\sharp\alpha'\smc \iota$ then $P(\alpha,\beta)\smc \iota=P(\alpha',\beta)\smc \iota,$ for all $\beta\in \Gamma(A^*).$

Moreover, $\Omega_L$ is nondegenerate. Note that if $X_L\in \Gamma(A_L),$
$$I\smc X_L=(P^\sharp\alpha)\smc \iota$$
and $\Omega_L(X_L,Y_L)=0,$
for all $Y_L\in \Gamma(A_L)$, then $P^\sharp\al\smc
\iota=0$ and therefore $X_L=0$. Hence, $\Omega_L$ is an
almost symplectic structure on $A_L$.

In order to show that $\Omega_L$ is symplectic, we will prove the
following Lemma.

\begin{lem}\label{prop:one:to:one:correspondence}
Let $X_L,Y_L$ be sections of $A_L$ and $\al,\be\in\Gamma(A^\ast)$
such that $I\smc X_L= P^\sharp\al \smc \iota$ and $I\smc Y_L=
P^\sharp\be \smc \iota$. Then:
\begin{itemize}
\item[$(i)$] $\Omega_L^{\flat}(X_L)=-I^\ast\al$,
\item[$(ii)$]
$i_{\brr{X_L,Y_L}_{A_L}}\Omega_L=\lie^{A_L}_{X_{L}}\be_L-\lie^{A_L}_{Y_{L}}\al_L+
\d^{A_L}(P(\al,\be)\smc \iota),$
\end{itemize}
where $\al_L=i_{X_L}\Omega_L$ and $\be_L=i_{X_L}\Omega_L$.
\end{lem}

\begin{proof}
$(i)$ If $Y_L\in\Gamma(A_L)$ is a section of $A_L$ which
$I$-projects on $P^\sharp\be$, for some $\be\in \Gamma(A^\ast)$, then
\begin{align*}
\Omega_L^\flat(X_L)(Y_L)=(\be\smc \iota)(P^\sharp\al \smc
\iota)=-(\al\smc \iota)(P^\sharp\be \smc \iota) =-(\al\smc
\iota)(I \smc Y_L)=-I^\ast\al(Y_L).
\end{align*}

$(ii)$ Note that, since $(I,\iota)$ and $P^\sharp$ are Lie
algebroid morphisms, we have
\begin{equation*}
I\smc \brr{X_L,Y_L}_{A_L}=\brr{P^\sharp\al, P^\sharp\be}_{A}\smc
\iota= P^\sharp\brr{\al,\be}_{P}\smc \iota.
\end{equation*}
So, by $(i)$ we obtain
\begin{equation}\label{eq:I-relation}
i_{\brr{X_L,Y_L}_{A_L}}\Omega_L=-I^\ast \brr{\al,\be}_{P}.
\end{equation}
Now, from  \eqref{eq:Koszul:bracket}, \eqref{eq:Lie:reduced} and
\eqref{eq:I-relation} we obtain the claim.
\end{proof}

\begin{prop}\label{prop:Omega_L:symplectic}
The $2$-section $\Omega_L$ on $A_L$ defined by \eqref{eq:omega:L}
is symplectic.
\end{prop}
\begin{proof}
We have only to prove that $\Omega_L$ is closed. In fact, for any $X_L,Y_L
\in \Gamma(A_L)$, we have

\begin{equation}\label{omega:closed}
\begin{split}
i_{X_L}i_{Y_L}\d^{A_L}\Omega_L &=i_{X_L}\lie^{A_L}_{Y_{L}}\Omega_L - i_{X_L}\d^{A_L}i_{Y_{L}}\Omega_L\\
&=\lie^{A_L}_{Y_{L}}i_{X_L}\Omega_L +
i_{\brr{X_L,Y_L}_{A_L}}\Omega_L- i_{X_L}\d^{A_L}i_{Y_{L}}\Omega_L,
\end{split}
\end{equation}
where we have used \eqref{eq:Cartan:formula} and
\eqref{wellknown}.

By applying Lemma \ref{prop:one:to:one:correspondence}, from
\eqref{omega:closed} we get
$$
i_{X_L}i_{Y_L}\d^{A_L}\Omega_L = \d^{A_L}i_{X_{L}}\be_L + \d^{A_L} (P(\al,\be)\smc \iota)=0.
$$
\end{proof}

Now, we consider a Nijenhuis operator $N:A\to A$ on the Lie
algebroid $A$ which is compatible with the Poisson structure $P$.
Using the compatibility condition $N\smc P^\sharp=P^\sharp\smc N^*,$ we may induce by restriction a
new operator $N_L:A_L\to A_L$ on $A_L$ such that
\begin{equation}\label{eq:N_L:N}
I\smc N_L(X_L)=N(P^\sharp \al)\smc\iota, \mbox{ for all $X_L\in \Gamma(A_L)$}
\end{equation}
where  $\al\in\Gamma(A^\ast)$ is a section of $A^*$ such that $X_L$
$I$-projects on $P^\sharp \al$.

Note that, from (\ref{eq:N_L:N}), we deduce that
\begin{equation}\label{24'}
I\smc N_L=N\smc I
\end{equation}
 which implies that
\begin{equation}\label{eq:N*I*}
N_L^\ast(I^\ast \alpha)=I^\ast (N^\ast \alpha), \mbox{ for }
\alpha\in \Gamma(A^\ast).
\end{equation}

\begin{thm}\label{thm:SN:Lie:algebroid}
Let $(A, P, N)$ be a Poisson-Nijenhuis Lie algebroid such that the
Poisson structure has constant rank in the leaves of the foliation
$D=\rho_{A}(P^\sharp(A^\ast))$. Then, we have a
symplectic-Nijenhuis Lie algebroid $(A_L, \Omega_L, N_L)$ on each
leaf $L$ of $D$.
\end{thm}
\begin{proof}
From Proposition \ref{prop:Omega_L:symplectic}, we deduce that
$(A_L,\Omega_L)$ is a symplectic Lie algebroid. Denote by $P_L$ the
Poisson structure corresponding to $\Omega_L$, defined by
$P_L^\sharp:= -(\Omega_L^\flat)^{-1}$. Note that, using Lemma \ref{prop:one:to:one:correspondence}, we have that
\begin{equation}\label{25'} P_L(I^*\alpha,I^*\beta)=P(\alpha,\beta)\smc \iota,\;\;\;\mbox{ for all }\alpha,\beta\in \Gamma(A^*).
\end{equation}

 Next, we prove that $N_L$ is a Nijenhuis operator compatible
with $P_L$. Indeed, firstly consider $X_L,Y_L$ sections of $A_L$.
Then, there are $\al$ and $\be$ sections of $A^\ast$ such that $X_L$
and $Y_L$ $I$-project on $P^\sharp \al$ and $P^\sharp \be,$
respectively. Thus, using \eqref{24'} and the fact that $(I,\iota)$ is a monomorphism of Lie algebroids, we
deduce that
\begin{equation}\label{eq:T_N_L:T_N}
I \smc \T_{N_L}(X_L,Y_L)= \T_{N}(P^\sharp\al, P^\sharp\be) \smc
\iota=0.
\end{equation}

On the other hand, for  $\al\in\Gamma(A^\ast),$ we consider the section
$X_L\in\Gamma(A_L)$ defined by
\[
I\smc X_L=P^\sharp \al \smc \iota.
\]
Using Lemma \ref{prop:one:to:one:correspondence} we deduce that
\begin{equation}\label{eq:P_L}
P_L^\sharp(I^\ast\al)= X_L.
\end{equation}

Now, from \eqref{eq:N_L:N}  and since $N \smc P^\sharp=P^\sharp \smc N^*,$ it follows that
$$I(N_L(X_L))=P^\sharp(N^*\alpha)\smc \iota.$$
Therefore, using again Lemma \ref{prop:one:to:one:correspondence}, we obtain that
$$P_L^\sharp(I^*(N^*\alpha))=N_L(X_L)=N_L(P^\sharp_L(I^*\alpha))$$
which implies that (see \eqref{eq:N*I*})
$$P^\sharp_L(N_L^*(I^*\alpha))=N_L(P_L^\sharp(I^*\alpha)).$$
This proves that $P_L^\sharp \smc N_L^*=N_L\smc P^\sharp_L.$

Finally, from \eqref{15'}, \eqref{15''}, \eqref{eq:Lie:reduced}, \eqref{eq:N_L:N}, \eqref{25'} and using that $N\smc P^\sharp=P^\sharp\smc N^*$ and the fact that $(I,\iota)$ is a Lie algebroid monomorphism, we conclude that
$$0=I^*(C(P,N)(\alpha,\beta))=C(P_L, N_L)(I^*\alpha, I^*\beta)\smc \iota,$$
for $\alpha,\beta\in \Gamma(A^*).$

This ends the proof of the result.

\end{proof}

%%%%%%%%%%%%%%%%%%%%%%%%%%%%%%%%%%%%%%%%%%%%%%%%%%%%%%%%%%%%%%%%%%%%%%%%%%%%%%%%%%%%%%%%%%
%%%%%%%%%%%%%%%%%%%%%%%%%%%%%%%%%%%%%%%%%%%%%%%%%%%%%%%%%%%%%%%%%%%%%%%%%%%%%%%%%%%%%%%%%%
\section{Reduction of Poisson-Nijenhuis Lie algebroids by epimorphisms of Lie algebroids}%
\label{sec:reduction:by:epimorphism}                                                     %
%%%%%%%%%%%%%%%%%%%%%%%%%%%%%%%%%%%%%%%%%%%%%%%%%%%%%%%%%%%%%%%%%%%%%%%%%%%%%%%%%%%%%%%%%%
%%%%%%%%%%%%%%%%%%%%%%%%%%%%%%%%%%%%%%%%%%%%%%%%%%%%%%%%%%%%%%%%%%%%%%%%%%%%%%%%%%%%%%%%%%
In order to complete the process of reduction, we now deal with the
general problem of the projectability of a Poisson-Nijenhuis structure
on a Lie algebroid with respect to a  vector bundle epimorphism.

Let $\tau_A \colon A \to M$ and $\tau_{\widetilde{A}} \colon
\widetilde{A} \to \widetilde{M}$ be vector bundles on the manifolds
$M$ and $\widetilde{M}$, respectively, and let $(\Pi,\pi)$  be an
epimorphism of vector bundles,
\[%%%COMMUTATIVE DIAGRAM
\xymatrix@C=4pc@R=4pc{ A \ar[r]^{\Pi}\ar[d]_{\tau_A} & \widetilde{A} \ar[d]^{\tau_{\widetilde{A}}}\\
                       M\ar[r]^{\pi} & \widetilde{M}}
\]
\vspace{3mm}

\noindent i.e., the map $\pi\colon M \to \widetilde{M}$ is a surjective
submersion and, for each $x\in M$, $\Pi_x\colon A_x \to
\widetilde{A}_{\pi(x)}$ is an epimorphism of vector spaces.

Denote by $\Gamma_p(A)$ (respectively, $\Gamma_p(A^\ast)$) the
space of the $\Pi$-projectable sections of $A$ (respectively, of
$A^\ast$). In \cite{IMMMP} a characterization is found to
establish when a vector bundle epimorphism is a Lie algebroid
epimorphism.

\begin{prop}\label{prop:epi} (see \cite{IMMMP})
Let $(\Pi,\pi):A\to \widetilde{A}$ be a vector bundle epimorphism.
Suppose that $(\br_A,\rho_A)$ is a Lie algebroid structure over
$A$. Then, there exists a unique Lie algebroid structure on
$\widetilde{A}$ such that $(\Pi,\pi)$ is a Lie algebroid
epimorphism if and only if the following conditions hold:
\begin{enumerate}
\item[i)] The space $\Gamma_p(A)$ of the $\Pi$-projectable
sections of $A$ is a Lie subalgebra of $(\Gamma(A),\br_A)$ and
\item[ii)] $\Gamma(\ker \Pi)$ is an ideal of
$\Gamma_p(A).$
\end{enumerate}
\end{prop}

In such a  case, the structure of Lie algebroid over $\widetilde{A}$
is characterized by
\begin{equation}\label{eq:projected:algebroid}
[\widetilde{X}, \widetilde{Y}]_{\widetilde{A}}\smc \pi=
\Pi\smc\brr{X,Y}_A, \qquad \rho_{\widetilde{A}}(\widetilde{X})
(\widetilde{f})\smc \pi= \rho_A(X)(\widetilde{f}\smc \pi),
\end{equation}
where $\widetilde{X}, \widetilde{Y}\in\Gamma(\widetilde{A})$,
$\widetilde{f}\in C^\infty(\widetilde{M})$ and $X,Y\in \Gamma(A)$
are such that
\[
\widetilde{X}\smc \pi=\Pi\smc X, \qquad \widetilde{Y}\smc \pi=\Pi \smc Y.
\]
Note that the real function $\rho_A(X)(\widetilde{f}\smc \pi)$ on $M$ is basic with respect to $\pi$ (see \cite{IMMMP}).

Let $(A, \br_A,\rho_A)$ and $(\widetilde{A},
\br_{\widetilde{A}},\rho_{\widetilde{A}})$ be Lie algebroids over
$M$ and $\widetilde{M}$, respectively, and let $(\Pi,\pi):A \to
\widetilde{A}$ be an epimorphism of Lie algebroids. We denote by
$V\pi$ the vertical subbundle of $\pi\colon M \to \widetilde{M}$.
Then, $\rho_A(Ker \Pi)\subseteq V\pi$ (see
(\ref{eq:morphism:anchors})).

We can always find a local basis $\{\xi_i,X_a\}$ of sections of
$A$ such that $\xi_i\in\Gamma(Ker \Pi),$  for all $i$, and  $X_a$ is a
$\Pi$-projectable section, for all $a$. Indeed, to obtain such a base we
choose a bundle  metric on $A$ which gives us the decomposition
$A=Ker\Pi\oplus (Ker\Pi)^\perp$ where $(Ker\Pi)^\perp$ is the
orthogonal complement defined by the chosen metric. Then we
consider a local basis $\{\xi_i\}$ of sections of $Ker\Pi$ and a local basis
$\{\widetilde{X}_a\}$ of sections of $\widetilde{A}$. It follows
that $\{\xi_i,X_a=\widetilde{X}_a^H \}$, where $\widetilde{X}_a^H$
is the horizontal lift of $\widetilde{X}_a$, is a local basis of sections
of $A$. Furthermore, note that if $\{\eta_i,\al_a\}$ is the dual
basis of $\{\xi_i,X_a\}$, then $\al_a=\Pi^\ast\widetilde{\al}_a$,
where $\{\widetilde{\al}_a\}$ is the dual basis of
$\{\widetilde{X}_a \}$ in  $\widetilde{A}$. By using these tools
we can prove the following results about projectable sections of
$A$ and $A^\ast$.

\begin{prop}\label{prop:projectable:sections}
Let $(\Pi,\pi):A \to
\widetilde{A}$ be an epimorphism of Lie algebroids and suppose that $X\in \Gamma(A)$ and $\alpha\in \Gamma(A^*)$.  Then,
\begin{enumerate}
\item[i)] If $X$ is  a $\Pi$-projectable section of $A$, then $\brr{\xi,X}_A\in
\Gamma(Ker\Pi)$ for any $\xi\in \Gamma(Ker\Pi)$. Moreover, if
$\al$ is a $\Pi$-projectable section of $A^*$, then $\al(\xi)=0$ and
$\lie_{\xi}^A\al=0,$ for any $\xi\in \Gamma(Ker\Pi)$.
\item[ii)] Assume that $\rho_A(Ker \Pi) = V\pi$. Then,
\begin{enumerate}
\item[a)] $X$ is a $\Pi$-projectable section of $A$ if and only if $\brr{\xi,X}_A\in
\Gamma(Ker\Pi),$ for any $\xi\in \Gamma(Ker\Pi)$.
\item[b)] $\al$ is a $\Pi$-projectable section of $A^*$ if and only if $\al(\xi)=0$ and $\lie_{\xi}^A\al=0,$ for any $\xi\in
\Gamma(Ker\Pi)$.
\end{enumerate}
\end{enumerate}
\end{prop}
\begin{proof}
The first part of $i)$ is a consequence of Proposition
\ref{prop:epi}.

Assume that there exists $\widetilde{\al}\in
\Gamma(\widetilde{A}^\ast)$ such that
$\al=\Pi^\ast\widetilde{\al}$. If $\xi\in \Gamma(Ker \Pi)$ then
$\al(\xi)=\Pi^\ast\widetilde{\al}(\xi)=0$ and, by using
\eqref{eq:Lie:reduced},
\[
\lie_{\xi}^A {\al} =\lie _{\xi}^A \Pi^\ast
\widetilde{\al}=0.
\]

To prove $ii)$ we proceed as follows. Let $\{\xi_i,X_a\}$ be a local
basis of sections of $A$ such that $\xi_i\in\Gamma(Ker \Pi),$ for all $i$, and $X_a$ is a $\Pi$-projectable section over $\widetilde{X}_a\in \Gamma(\widetilde{A})$ for all $a.$
\begin{enumerate}
\item[$a)$] Suppose that $X\in \Gamma(A)$ is such that $[\xi,X]_A\in \Gamma(\ker \Pi),$ for any $\xi\in \Gamma(\ker\Pi).$ If
\[
X=f^i\xi_i+F^a X_a\qquad \mbox{with }f^i,F^a \mbox{ local $C^\infty$-functions on } M
\]
then, by using Proposition
\ref{prop:epi}, we have that
\[
0=\Pi\smc\brr{\xi_i,X}_A=\Pi\smc(\rho_A(\xi_i)(F^a) X_a)=\rho_A(\xi_i)(F^a)(\widetilde{X}_a\smc \pi).
\]
So, if $Z\in V_x\pi$, with $x\in M,$ then there exists $\xi\in\Gamma(Ker \Pi)$
such that $Z=\rho_A(\xi)(x)$ and therefore
\[
Z(F^a)=\rho_A(\xi)(F^a)(x)=0.
\]
We conclude that there exists $\widetilde{F}^a\in
C^\infty(\widetilde{M})$ such that
\[
F^a=\widetilde{F}^a\smc\pi,
\]
and $X$ is a $\Pi$-projectable section of $A$.

\item[$b)$] Assume  that $\al$ is a section of $A^\ast$ such that $\al(\xi)=0$ and $\lie_{\xi}^A\al=0$, for any $\xi\in
\Gamma(Ker\Pi)$. Let $\{\eta_i,\Pi^\ast\widetilde{\al}_a\}$ be the
dual basis of $\{\xi_i,X_a\}$. Thus,
\[
\al=g^i\eta_i + \sigma^a \Pi^\ast\widetilde{\al}_a, \qquad
\mbox{with }g^i,\sigma^a\in C^\infty(M).
\]

As $\al(\xi_i)=0$, we deduce that $g^i=0$. On the other hand, using (\ref{wellknown}) and Proposition \ref{prop:epi},
\[
0=\lie^A_{\xi_i}\al(X_a)=\rho_A(\xi_i)(\sigma^a)-
\al(\brr{\xi_i,X_a}_A) =\rho_A(\xi_i)(\sigma^a).
\]
As before, this implies that
$\sigma^a=\widetilde{\sigma}^a\smc\pi$ for some function
$\widetilde{\sigma}^a\in C^\infty(\widetilde{M})$. Hence, $\al$ is
$\Pi$-projectable.
\end{enumerate}
\end{proof}

We consider now a section $P$ of the vector bundle $\wedge^2A \to M$.
$P$ is said to be $\Pi$-\emph{projectable} if,  for each
$\widetilde{\al}\in\Gamma(\widetilde{A}^\ast),$ we have
$P^\sharp\Pi^\ast\widetilde{\al}\in\Gamma_p(A)$.

\begin{prop}\label{prop:projectable:bivector}
Let $(\Pi,\pi):A \to
\widetilde{A}$ be an epimorphism of Lie algebroids.
If $P\in\Gamma(\wedge^2 A)$ is $\Pi$-projectable, then
\begin{equation}\label{eq:lie:P}
([\xi,P]_A)^\sharp(\Gamma_p(A^\ast))\subseteq\Gamma(Ker\Pi)
\end{equation}
for any $\xi\in\Gamma(Ker\Pi)$. Moreover, if
$\rho_A(Ker\Pi)=V\pi$, then $P$ is $\Pi$-projectable if and only
if \eqref{eq:lie:P} holds.
\end{prop}

\begin{proof}

Assume that $P$ is $\Pi$-projectable. Then, for any $\al\in
\Gamma_p(A^\ast)$ and $\xi\in\Gamma(Ker\Pi)$, by using
\eqref{eq:P:sharp} and Proposition \ref{prop:projectable:sections}
we have
\begin{align*}
([\xi, P]_A)^\sharp (\al)&=[\xi, P^\sharp \al]_A- P^\sharp
\lie_\xi^A \al\\
&=[\xi, P^\sharp \al]_A \in\Gamma(Ker \Pi).
\end{align*}

Now, we suppose that $P$ satisfies \eqref{eq:lie:P} and
$\rho_A(Ker\Pi)=V\pi$. Consider a local basis of sections
$\{\xi_i,X_a\}$ of $A$ such that $\xi_i\in\Gamma(Ker \Pi)$ and
$X_a\in\Gamma_p(A)$. Let $\{\eta_i,\Pi^\ast\widetilde{\al}_a\}$ be
the dual basis of $\{\xi_i,X_a\}$. We have
\[
P^\sharp \Pi^\ast\widetilde{\al}_a=f^i_a\xi_i+F^b_a X_b,\qquad
\mbox{with } f^i_a,F^b_a\mbox{ local real  $C^\infty$-functions on $M.$}
\]
Note that $F^b_a=-F^a_b$.

By using Proposition \ref{prop:projectable:sections} we have
\begin{align*}
0=&\Pi\smc ( ([\xi, P]_A)^\sharp (\Pi^\ast\widetilde{\al}_a)
)(\widetilde{\al}_b)=([\xi, P]_A)(\Pi^\ast\widetilde{\al}_a,
\Pi^\ast\widetilde{\al}_b)\\
&= \rho_A(\xi)(P(\Pi^\ast\widetilde{\al}_a,
\Pi^\ast\widetilde{\al}_b)) = \rho_A(\xi)(F_a^b),
\end{align*}
for any $\xi\in\Gamma(Ker\Pi)$.

So, if $Z\in V_x\pi$, then there exists $\xi\in\Gamma(Ker
\Pi)$  such that $Z=\rho_A(\xi)(x)$ and therefore
\[
Z(F_a^b)=0.
\]
Hence, there exists a local real $C^\infty$-function $\widetilde{F_a^b}$ on $\widetilde{M}$
such that
\[
F_a^b=\widetilde{F_a^b}\smc\pi.
\]
\end{proof}

If $P$ is a $\Pi$-projectable Poisson structure on $A$, then we
may construct the $2$-section $\widetilde{P}\in\Gamma(\wedge^2
\widetilde{A})$ of $\widetilde{A}$ characterized by
\begin{equation}\label{eq:projection:Poisson:tensor}
(\widetilde{P}^\sharp\widetilde{\al})\smc \pi= \Pi(P^\sharp
(\Pi^\ast\widetilde{\al})), \mbox{ for any } \widetilde{\al}\in\Gamma(\widetilde{A}^\ast)
\end{equation}
or equivalently,
\begin{equation}\label{31'}
\widetilde{P}(\widetilde{\alpha},\widetilde{\beta})\smc \pi=P(\Pi^*\widetilde{\alpha},\Pi^*\widetilde{\beta}), \mbox{ for any } \widetilde{\al},\widetilde{\be}\in\Gamma(\widetilde{A}^\ast).
\end{equation}

\begin{prop}\label{prop:projectable:Poisson}
Let $(\Pi,\pi):A \to
\widetilde{A}$ be an epimorphism of Lie algebroids. If $P$ is a $\Pi$-projectable Poisson structure on $A$, then
$\widetilde{P}$ is a Poisson structure on $\widetilde{A}$.
\end{prop}
\begin{proof}
Let $\widetilde{\al}\in\Gamma(\widetilde{A}^\ast)$. Then, by using
\eqref{eq:P:sharp} one may prove that
\begin{equation}\label{eq:condition:poisson:tensor:tilde}
\frac{1}{2}
i_{\widetilde{\al}}[\widetilde{P},\widetilde{P}]_{\widetilde{A}}=
-\widetilde{P}^\sharp (\d^{\widetilde{A}}\widetilde{\al}) +
[\widetilde{P}^{\sharp} \widetilde{\al},
\widetilde{P}]_{\widetilde{A}}
\end{equation}
where $\widetilde{P}^\sharp (\d^{\widetilde{A}}\widetilde{\al})$
is the section of the vector bundle $\wedge^2\widetilde{A}\to \widetilde{M}$ defined by
\[
\widetilde{P}^\sharp (\d^{\widetilde{A}}\widetilde{\al})
(\widetilde{\be_1},\widetilde{\be_2}) = \d^{\widetilde{A}}
\widetilde{\al} (\widetilde{P}^\sharp \widetilde{\be}_1,
\widetilde{P}^\sharp \widetilde{\be}_2),
\]
for any $\widetilde{\be}_1, \widetilde{\be}_2
\in\Gamma(\widetilde{A}^\ast)$.

From the equality \eqref{eq:condition:poisson:tensor:tilde} for
the Poisson structure $P$ and the $1$-section $\Pi^*\widetilde{\al}$ of $A$, we deduce that
\begin{equation}\label{eq:condition:poisson:tensor}
P^\sharp (\d^A \Pi^\ast\widetilde{\al}) =
[P^\sharp(\Pi^\ast\widetilde{\al}), P]_A.
\end{equation}
On the other hand, from \eqref{eq:Lie:alg:morphism} and
\eqref{eq:projection:Poisson:tensor} we deduce that
\begin{equation}\label{eq:project:Pi:poisson:condition}
\wedge^2\Pi\smc P^\sharp\d^A \Pi^\ast\widetilde{\al} =
\widetilde{P}^\sharp \d^{\widetilde{A}} \widetilde{\al}\smc \pi.
\end{equation}
Projecting by $\Pi$, the equation
\eqref{eq:condition:poisson:tensor} and using
\eqref{eq:project:Pi:poisson:condition} we get
\begin{equation}\label{eq:Poisson:Lie:Poisson}
\widetilde{P}^\sharp \d^{\widetilde{A}} \widetilde{\al}\smc
\pi=\wedge^2\Pi\smc[P^\sharp \Pi^\ast\widetilde{\al},  P]_A.
\end{equation}
Since $(\Pi,\pi)$ is an epimorphism of Lie algebroids,  from
\eqref{eq:Lie:reduced} and \eqref{eq:projection:Poisson:tensor} we
also obtain
\begin{equation}\label{eq:lie:commutes:Pi}
\lie^A_{P^\sharp (\Pi^\ast\widetilde{\al})} (\Pi^\ast\widetilde{\be})
=\Pi^\ast(\lie^{\widetilde{A}}_{\widetilde{P}^\sharp
\widetilde{\al}} \widetilde{\be}) \qquad\mbox{for any }
\widetilde{\be}\in\Gamma(\widetilde{A}^\ast).
\end{equation}
\begin{equation}\label{36'}
\begin{array}{rcl}
\rho_A(P^\sharp(\Pi^*\widetilde{\al}))(\widetilde{f}\smc \pi)&=&\lie^A_{P^\sharp(\Pi^*\widetilde{\al})}(\widetilde{f}\smc \pi)=(\lie^{\widetilde{A}}_{\widetilde{P}^\sharp\widetilde{\al}}\widetilde{f})\smc \pi\\[5pt]&=&\rho_{\widetilde{A}}(\widetilde{P}^\sharp\widetilde{\al})(\widetilde{f})\smc \pi, \;\;\; \mbox{ with }\widetilde{f}\in C^\infty(\widetilde{M}).
\end{array}
\end{equation}
This fact allows us to prove that
\begin{equation}\label{eq:lie:projection:poisson}
\wedge^2\Pi\smc [P^\sharp(\Pi^\ast\widetilde{\al}), P]_A =
[\widetilde{P}^\sharp \widetilde{\al},
\widetilde{P}]_{\widetilde{A}} \smc \pi,
\end{equation}
by using \eqref{9'}.

From \eqref{eq:condition:poisson:tensor:tilde},
\eqref{eq:Poisson:Lie:Poisson} and
\eqref{eq:lie:projection:poisson} we deduce that
\[
i_{\widetilde{\al}}[\widetilde{P},\widetilde{P}]_{\widetilde{A}}=0,
\]
for any $\widetilde{\al}\in\Gamma(\widetilde{A}^\ast)$. In
conclusion $\widetilde{P}$ is a Poisson structure.
\end{proof}

Assume that $N\colon A\to A$ is a Nijenhuis operator on $A$. %compatible with $P$
 $N$ is said to be $\Pi$-\emph{projectable} if
\[
N(\Gamma_p(A))\subseteq \Gamma_p(A) \quad \mbox{and} \quad
N(\Gamma(Ker\Pi))\subseteq \Gamma(Ker\Pi).
\]

\begin{prop}\label{prop:projectable:N}
Let $(\Pi,\pi):A \to
\widetilde{A}$ be an epimorphism of Lie algebroids.
If $N$ is a $\Pi$-projectable Nijenhuis operator on $A$, then
\begin{equation}\label{eq:lie:N}
\lie^A_{\xi} N(\Gamma_p(A))\subseteq \Gamma(Ker\Pi) \quad
\mbox{and} \quad N(\Gamma(Ker\Pi))\subseteq \Gamma(Ker\Pi),
\end{equation}
for any $\xi\in\Gamma(Ker\Pi)$. Moreover, if
$\rho_A(Ker\Pi)=V\pi$, then $N$ is $\Pi$-projectable if and only
if \eqref{eq:lie:N} holds.
\end{prop}
\begin{proof}
Assume that $N$ is $\Pi$-projectable. For any $X\in\Gamma_p(A)$
\[
(\lie^A_\xi N)(X)=[\xi,NX]_A-N([\xi,X]_A).
\]
Since $NX\in\Gamma_p(A)$, $[\xi,NX]_A \in \Gamma(Ker\Pi)$ and
$[\xi,X]_A \in \Gamma(Ker\Pi)$ (see Proposition \ref{prop:epi}),
then
\[
\lie^A_{\xi} N(X)\in \Gamma(Ker\Pi).
\]

Now we suppose that $\rho_A(Ker\Pi)=V\pi$ and that \eqref{eq:lie:N} holds.
 Consider a local basis of sections $\{\xi_i,X_a\}$ of
$A$ such that $\xi_i\in\Gamma(Ker \Pi)$ and $X_a\in\Gamma_p(A)$.
Then, $X\in\Gamma_p(A)$, implies that $N(X)\in\Gamma_p(A)$.
Indeed,
\[
N(X)=f^i\xi_i+F^a X_a\qquad \mbox{with }f^i,F^a\mbox{ local real $C^\infty$-functions on $M$.}
\]
Hence, keeping in account that $N[\xi,X]_A\in\Gamma(Ker \Pi)$, we
have
\begin{align*}
\begin{array}{rcl}
0&=&\Pi\smc (\lie^A_{\xi}N(X))= \Pi\smc ([\xi,NX]_A-N([\xi,X]_A))
\\&=&\Pi\smc ([\xi,NX]_A)=\Pi\smc (\rho_A(\xi)(F^a)X_a).
\end{array}
\end{align*}
Therefore, $\rho_A(\xi)(F^a)=0$ for any $\xi\in\Gamma(Ker \Pi)$.

Let $Z\in V_x\pi$, with $x\in M.$ Hence, there exists $\xi\in\Gamma(Ker
\Pi)$ such that
\[
Z=\rho_A(\xi)(x).
\]
Thus, we can conclude that $Z(F^a)=0$, i.e. there exists a local $C^\infty-$function $\widetilde{F_a}$ on $\widetilde{M}$
 such that
\[
F_a=\widetilde{F_a}\smc\pi.
\]
\end{proof}

If $N$ is a $\Pi$-projectable Nijenhuis operator on $A$, then we
can construct a new operator $\widetilde{N}\colon \widetilde{A}\to
\widetilde{A}$ as follows.
\begin{equation}\label{eq:N:tilde}
(\widetilde{N}\widetilde{X})\smc \pi=\Pi\smc (NX) \qquad \mbox{for any
} \widetilde{X}\in\Gamma(\widetilde{A}),
\end{equation}
where $X\in\Gamma_p(A)$ is a projectable section such that $\Pi\smc
X=\widetilde{X}\smc\pi$. Note that $\widetilde{N}$ is well defined
since $X\in\Gamma_p(A)$ and therefore $NX\in\Gamma_p(A)$.
Moreover, if $X'$ is another section of $A$ such that $\Pi\smc X'=\Pi\smc X$ then $X'-X\in \Gamma(Ker \Pi)$ and $NX=NX'.$

From previous results, we give conditions for obtaining a Poisson-Nijenhuis structure on the Lie algebroid image of a Lie algebroid epimorphism.
\begin{thm}\label{thm:PN:epi}
Let $(\Pi,\pi):A\to \widetilde{A}$ be a Lie algebroid epimorphism.
Assume that $(P,N)$ is a Poisson-Nijenhuis structure on $A$ such
that $P$ and $N$ are $\Pi$-projectable. Then,
$(\widetilde{P},\widetilde{N})$ is  a Poisson-Nijenhuis structure
on $\widetilde{A}$.
\end{thm}
\begin{proof}
We will show that $\widetilde{N}$ is compatible with the Poisson
structure $\widetilde{P}$. Indeed, firstly we show that
\begin{equation}\label{eq:NP=PN:tilde}
\widetilde{N}\smc\widetilde{P}^\sharp=\widetilde{P}^\sharp\smc\widetilde{N}^\ast.
\end{equation}
From \eqref{eq:projection:Poisson:tensor} and \eqref{eq:N:tilde}
it follows that for any
$\widetilde{\al}\in\Gamma(\widetilde{A^\ast})$
\begin{align*}
\widetilde{N}(\widetilde{P}^\sharp\widetilde{\al})\smc\pi=\Pi\smc
N(P^\sharp\Pi^\ast\widetilde{\al})
=\Pi \smc P^\sharp N^\ast(\Pi^\ast\widetilde{\al})=\Pi \smc P^\sharp (\Pi^*\widetilde{N}^*\widetilde{\al})=
\widetilde{P}^\sharp (\widetilde{N}^\ast \widetilde{\al})
\smc \pi.
\end{align*}
On the other hand, using \eqref{eq:Nijenhuis}, \eqref{eq:deformed:bracket}  and the fact that $(\Pi,\pi)$ is a Lie algebroid morphism, we get
\begin{equation}\label{eq:T_N:tilde}
\T_{\widetilde{N}}(\widetilde{X},\widetilde{Y})\smc \pi=\Pi\smc
\T_N(X,Y) \qquad \mbox{for any } \widetilde{X},\widetilde{Y}
\in\Gamma(\widetilde{A}),
\end{equation}
where $X,Y\in\Gamma(A)$ are such that $\widetilde{X}\smc \pi=\Pi\smc
X$, $\widetilde{Y}\smc \pi=\Pi \smc Y$.

Finally, by using \eqref{eq:Lie:alg:morphism:same:base}, \eqref{eq:Lie:reduced},  \eqref{eq:projection:Poisson:tensor}  and
\eqref{eq:N:tilde},  we can prove that
\begin{equation*}\label{eq:bracket:NP:tilde}
\Pi^\ast[\widetilde{\al},\widetilde{\be}]_{\widetilde{P}}
=[\Pi^\ast\widetilde{\al},\Pi^\ast\widetilde{\be}]_{P},\;\;\;
\Pi^\ast[\widetilde{\al},\widetilde{\be}]_{\widetilde{N}\widetilde{P}}
=[\Pi^\ast\widetilde{\al},\Pi^\ast\widetilde{\be}]_{NP}
\end{equation*}
and
\begin{equation*}\label{eq:bracket:N:star:P:tilde}
\Pi^\ast[\widetilde{\al},\widetilde{\be}]^{\widetilde{N}^\ast}_{\widetilde{P}}
=[\Pi^\ast\widetilde{\al},\Pi^\ast\widetilde{\be}]_P^{N^\ast}.
\end{equation*}
As a consequence,
\begin{equation}\label{eq:C(P,N):tilde}
\Pi^\ast(\C(\widetilde{P},\widetilde{N})(\widetilde{\al},\widetilde{\be}))
=\C(P,N)(\Pi^\ast\widetilde{\al},\Pi^\ast\widetilde{\be}),
\end{equation}
for any
$\widetilde{\al},\widetilde{\be}\in\Gamma(\widetilde{A}^\ast)$.

From \eqref{eq:NP=PN:tilde}, \eqref{eq:T_N:tilde} and
\eqref{eq:C(P,N):tilde}
%From \eqref{eq:NP=PN:tilde}, \eqref{eq:T_N:tilde},
%\eqref{eq:bracket:NP:tilde} and \eqref{eq:bracket:N:star:P:tilde},
we obtain that $(\widetilde{P},\widetilde{N})$ is a
Poisson-Nijenhuis structure on $\widetilde{A}$.
\end{proof}

The above result suggests us to introduce the following definition.

\begin{defn}\label{def:PN:Lie:alg:morphism}
Let $(\Pi,\pi):A\to \widetilde{A}$ be a Lie algebroid morphism. We
say that $(\Pi,\pi)$ is a \emph{Poisson-Nijenhuis Lie algebroid
morphism} if we have Poisson-Nijenhuis structures $(P, N)$,
$(\widetilde{P}, \widetilde{N})$ on $A$ and $\widetilde{A}$,
respectively, such that
\begin{align*}
(\widetilde{P}^\sharp\widetilde{\al})\smc\pi&=
\Pi \smc (P^\sharp (\Pi^\ast\widetilde{\al})),\\
(\widetilde{N}\widetilde{X})\smc \pi &= \Pi \smc (NX),
\end{align*}
for all $\widetilde{\al}\in\Gamma(\widetilde{A}^\ast)$,
$X\in\Gamma(A)$ and $\widetilde{X}\in\Gamma(\widetilde{A})$ such
that $\widetilde{X}\smc \pi=\Pi\smc X$.
\end{defn}

The following result follows easily from Proposition
\ref{prop:epi} and Theorem \ref{thm:PN:epi}.
\begin{thm}\label{prop:PN:morphism}
Let $(\Pi,\pi):A\to \widetilde{A}$ be a vector bundle epimorphism.
Suppose that $(\br_A,\rho_A,P,N)$ is a Poisson-Nijenhuis Lie
algebroid structure over $A$. Then, there exists a unique Poisson-Nijenhuis Lie algebroid structure on $\widetilde{A}$ such that $(\Pi,\pi)$ is a
Poisson-Nijenhuis Lie algebroid epimorphism if and only if the
following conditions hold:
\begin{enumerate}
\item[i)] The space $\Gamma_p(A)$ of the $\Pi$-projectable
sections of $A$ is a Lie subalgebra of $(\Gamma(A),\br_A)$;
\item[ii)] $\Gamma(\ker \Pi)$ is an ideal of
$\Gamma_p(A)$ and
\item[iii)] $P$ and $N$ are $\Pi$-projectable.
\end{enumerate}
\end{thm}

%%%%%%%%%%%%%%%%%%%%%%%%%%%%%%%%%%%%%%%%%%%%%%%%%%%%%%%%%%%%%%%%%%%%%%%
%%%%%%%%%%%%%%%%%%%%%%%%%%%%%%%%%%%%%%%%%%%%%%%%%%%%%%%%%%%%%%%%%%%%%%%
\section{Reduction of a Lie algebroid induced by a Lie subalgebroid}  \label{section5}
\label{sec:reduction:subalgebroid}%%%%%%%%%%%%%%%%%%%%%%%%%%%%%%%%%%%%%
%%%%%%%%%%%%%%%%%%%%%%%%%%%%%%%%%%%%%%%%%%%%%%%%%%%%%%%%%%%%%%%%%%%%%%%
In this section we will describe, using the above results about  reduction by epimorphisms  of Lie algebroids,  the reduction of a Lie algebroid by a certain  foliation
associated with  a given
Lie subalgebroid.
In the next section, we will use this construction for obtaining, under suitable regularity conditions,  a
reduced nondegenerate Poisson-Nijenhuis Lie algebroid from an
arbitrary Poisson-Nijenhuis Lie algebroid through a suitable choice of the Lie subalgebroid.

In this reduction procedure of a Lie algebroid,   fundamental tools  are the complete and vertical lifts of sections associated with a Lie algebroid. Firstly, we recall these notions and some properties about them.

%%%%%%%%%%%%%%%%%%%%%%%%%%%%%%%%%%%%%%%%%
\subsection{Complete and vertical lifts in a Lie algebroid}%
%%%%%%%%%%%%%%%%%%%%%%%%%%%%%%%%%%%%%%%%%%%%%%%%%%%%%%%%%%%%
Let $(A,\br_A,\rho_A)$  be a Lie algebroid over a manifold $M$ and
$\tau_A:A\to M$ be the corresponding vector bundle projection.

Given $f\in C^\infty(M)$, we will denote by $f^c$ and $f^v$ {\it
the complete and vertical lift} to $A$ of $f$. Here $f^c$ and
$f^v$ are the real functions on $A$ defined by
\begin{equation}\label{fcv}
f^c(a)=\rho_{A}(a)(f),\;\;\;\; f^v(a)=f(\tau_A(a)),
\end{equation}
for all $a\in A$.

Now, let $X$ be a section of $A$. Then, we can consider {\it the
vertical lift of $X$ to $A$} as the vector field $X^v$ on $A$
given by
\[
X^v(a)=X(\tau_A(a))_a^v,\;\;\; \mbox{ for } a\in A,
\]
where $\;^v_{a}:A_{\tau_A(a)}\to T_a(A_{\tau_A(a)})$ is the
canonical isomorphism between the vector spaces $A_{\tau_A(a)}$
and $T_a(A_{\tau_A(a)})$.

On the other hand, there exists a unique vector field $X^c$ on $A$,
{\it the complete lift of $X$ to $A$,} characterized by the two
following conditions:
\begin{enumerate}
\item[$(i)$] $X^c$ is $\tau$-projectable on $\rho_{A}(X)$ and
\item[$(ii)$] $X^c(\hat{\alpha})=\widehat{{\mathcal L}^A_X\alpha},$
\end{enumerate}
for all $\alpha\in \Gamma(A^*)$ (see \cite{GU1}). Here, if
$\beta\in\Gamma(A^*)$ then $\hat\beta$ is the linear function on
$A$ defined by
\[
\hat\beta(a)=\beta(\tau_A(a))(a),\;\;\; \mbox{ for all  $a\in A$.}
\]

Complete and vertical lifts may be extended to associate to any
section $Q\colon M\to\wedge^q A$ of the bundle $\wedge^q A\to M$ a
pair of $q$-multivectors $Q^c,Q^v:A\to \wedge^q (TA)$ on $A.$ These
extensions are uniquely determined by the following equalities
(see \cite{GU1}):
\begin{equation}\label{eq:exterior:complete}
(Q\wedge R)^c = Q^c \wedge R^v + Q^v \wedge R^c,
\end{equation}
and
\begin{equation}\label{eq:exterior:vertical}
(Q\wedge R)^v = Q^v \wedge R^v,
\end{equation}
which are satisfied by any pair of sections $Q:M\to \wedge^q A$,
$R:M\to \wedge^r A$.

A direct computation proves that (see \cite{GU1})
\begin{equation}\label{eq:complete:vertical}
[Q^c,R^c]=[Q,R]_A^c,\;\;\; [Q^c,R^v]=[Q,R]_A^v, \;\;\;
[Q^v,R^v]=0.
\end{equation}

Given $X\in\Gamma(A)$, we can also define the complete lift of $X$
to $A^\ast$ as the vector field $X^{\ast c}$ over $A^\ast$ such
that it is $\tau_{A^\ast}$-projectable on $\rho_A(X)$ and
\begin{equation}\label{eq:completo:star}
X^{\ast c}(\widehat{Y})=\widehat{[X,Y]_A},
\end{equation}
for all $Y\in\Gamma(A)$ (see \cite{GU2}). Here $\widehat{Z}$, with
$Z\in\Gamma(A)$, is the linear map over $A^\ast$ induced by $Z$.
In fact, the complete lifts of a section $X\in\Gamma(A)$ to $A$
and $A^\ast$ are related by the following formula
\begin{equation}\label{eq:completo:star:flusso}
X^{\ast c}(\widehat{Y})=\frac{\d}{\d t}_{|t=0}(\widehat{Y}\smc
\varphi_t^{\ast}), \qquad   \mbox{for any }Y\in\Gamma(A)
\end{equation}
where $\varphi_t\colon A \to A$ is the flow of $X^c\in\X(A)$ (see
\cite{MX2,MAR}).

Suppose that $(x^i)$ are coordinates on an open subset $U$ of $M$,
$\{e_\alpha\}$ is a basis of sections of $\tau_A^{-1}(U)\to U$ and
$\{e^\alpha\}$ is the dual basis of sections of
$\tau_{A^*}^{-1}(U)\to U$. Denote by $(x^i,y^\alpha)$ the
corresponding local coordinates on $\tau_A^{-1}(U)$ and by
$(x^i,y_\alpha)$ the local coordinates on $\tau_{A^*}^{-1}(U)$. Finally,
let $\rho_\alpha^i$ and $C_{\alpha\beta}^\gamma$ be the
corresponding local structure functions of $A$, defined by
\[
\rho_A(e_\alpha)=\rho_\alpha^i\frac{\partial }{\partial
x^i}\;\;\;\mbox{ and }\;\;\; [
e_\alpha,e_\beta]_A=C_{\alpha\beta}^\gamma e_\gamma.
\]

If $X$ is a section of $A$ and on $U$ we have
\[
X=X^\alpha e_\alpha,
\]
then the coordinate expressions of the lifts are given by
\begin{equation}\label{eq:complete:vertical:coordinate}
\begin{array}{rcl} X^v&=&\displaystyle
X^\alpha\displaystyle\frac{\partial}{\partial y^\alpha},\\
X^c&=&X^\alpha \rho_\alpha^i\displaystyle\frac{\partial }{\partial
x^i} + \left(\rho_\beta^i\displaystyle\frac{\partial
X^\alpha}{\partial x^i}-X^\gamma
C_{\gamma\beta}^\alpha\right)y^\beta\displaystyle\frac{\partial
}{\partial y^\alpha},\\
X^{*c}&=&X^\alpha \rho_\alpha^i
\displaystyle\frac{\partial}{\partial x^i} -\left(\rho^i_\alpha
\frac{\partial X^\beta}{\partial x^i} y_\beta +
C^\gamma_{\alpha\beta}y_\gamma X^\beta\right)
\displaystyle\frac{\partial}{\partial y_\alpha}.
\end{array}
\end{equation}

In particular,
\begin{equation}\label{eq:e:complete:vertical:coordinate}
e_\alpha^v=\frac{\partial }{\partial y^\alpha},\;\;\;
e_\alpha^c=\rho_\alpha^i\frac{\partial }{\partial
x^i}-C_{\alpha\beta}^\gamma y^\beta\frac{\partial }{\partial
y^\gamma}, \;\;\;  e_\alpha^{*c}=\rho_\alpha^i\frac{\partial
}{\partial x^i} + C_{\alpha\beta}^\gamma y_\gamma\frac{\partial
}{\partial y_\beta}.
\end{equation}

%%%%%%%%%%%%%%%%

\subsection{Reduction procedure of a Lie algebroid induced by a Lie subalgebroid}

Before describing this procedure, we prove the following general lemma on vector bundles which will be useful in the
sequel.
\begin{lem}\label{induced:VB}
Let $\pi_A\colon A\to M$ a vector bundle of rank $k$ and  $\pi_B\colon
B\to M'$ be a surjective submersion. Assume that there exist two smooth maps
$\Phi\colon A\to B$ and $\phi\colon M\to M'$ in such a way that  the following
diagram
\[%%%COMMUTATIVE DIAGRAM
\xymatrix@C=3pc@R=3pc{A\ar[r]^{\pi_A}\ar[d]^{\Phi} & M\ar[d]^{\phi}\\
            B\ar[r]^{\pi_B} & M'}
\]
\vspace{3mm}

\noindent is commutative and such that
\begin{itemize}
\item[1)] $\phi$ is a submersion;
\item[2)] $\forall x\in M$, $\Phi_x\colon \pi_A^{-1}(x)\to \pi_B^{-1}(\phi(x))$ is
a diffeomorphism and
\item[3)] $\forall x,y\in M$ such that $\phi(x)=\phi(y)$,
\[
\Phi^{-1}_{y} \smc\Phi_x \colon \pi_A^{-1}(x)\to \pi_A^{-1}(y)
\]
is an isomorphism of vector spaces.

Then, $\pi_B\colon B\to M'$ is a vector bundle of rank $k$ and $(\Phi,\phi)$ is a vector bundle epimorphism.
\end{itemize}
\end{lem}
\begin{proof}
Let $x'=\phi(x)\in M'$. Then, there exists a unique structure of
vector space on the fiber $\pi_B^{-1}(x')$ in such a way that the
diffeomorphism $\Phi_x:\pi_A^{-1}(x)\to \pi_B^{-1}(x')$ is an
isomorphism of vector spaces. Moreover, this structure doesn't
depend on the chosen point $x\in M.$ In fact, if $y\in M$ and
$\phi(y)=\phi(x)=x'$ then, from the third assumption above, we
deduce that the map $\Phi_y^{-1}\smc \Phi_x:\pi_A^{-1}(x)\to
\pi_A^{-1}(y)$ is an isomorphism of vector spaces.

On the other hand, using that $\phi$ is a submersion and the fact
that $\pi_A:A\to M$ is a vector  bundle, we have that there exists
an open neighbourhood $U\subset M$ of $x\in M$,
 an open  neighbourhood $U'\subseteq M'$ of
$x'\in M'$ and two smooth maps  $s\colon U'\to
U$ and $\psi\colon U\times \Rr^k \to \pi_A^{-1}(U)$ such that
\begin{itemize}
\item[$1)$] $\phi\smc s =1_{U'}$ and $s(x')=x$.
\item[$2)$] $\psi$ is a diffeomorphism, $\pi_A \smc \psi =pr_1$ and for
each $y\in U$, $\psi_{y}\colon \Rr^k \to \pi_A^{-1}(y)$ is a
vector space isomorphism.
\end{itemize}
Therefore we can construct a diffeomorphism
\[
\overline{\psi}\colon U'\times \Rr^k \to \pi_B^{-1}(U')
\]
as follows: $\overline{\psi}(y',g)=(\Phi\smc \psi_{s(y')})(g)$ for $(y',g)\in U'\times {\Bbb R}^k.$  Note
that
$\overline{\psi}^{-1}(b)=\left(\pi_B(b),(\psi^{-1}_{s(\pi_B(b))}\smc
\Phi^{-1}_{s(\pi_B(b))})(b)\right)$. Moreover, if $y'\in U'$ it is easy to prove that
$\overline{\psi}_{y'}\colon \Rr^k \to \pi_B^{-1}(y')$
is an isomorphism of vector spaces.
\end{proof}

Let $\tau_A\colon A\to M$ be a vector bundle and
$\left(\br_{A},\rho_{A}\right)$ be a Lie algebroid structure on
$A$. Consider a Lie subalgebroid $\tau_B\colon B\to M$ of $A$.
Then, we have the following result.
\begin{prop}\label{prop:gen:foliation}
\noindent\begin{itemize}
\item[$1)$] The generalized distribution $\rho_A(B)$ on $M$ defined by
$$\rho_A(B)_x=\rho_A(B_x)\subseteq T_xM, \mbox{ for every } x\in M,$$
 is a generalized foliation. Moreover,
\[
\dim(\rho_A(B)_x)\leq \rank B, \mbox{ for every }x\in M.
\]
\item[$2)$] The generalized distribution $\mathcal{F}^B$ on $A$
defined by
\[
\mathcal{F}^B_a=\{X^c(a)+Y^v(a)\;|\; X,Y\in\Gamma(B)\}\subseteq T_aA, \;\; \mbox{ for $a\in A$ }
\]
is a generalized foliation. Furthermore,  $\dim \mathcal{F}^B_b=\dim (\rho_A(B))_{\tau_B(b)} + \rank B,$ for $b\in B.$  Thus, if $\mathcal{F}^B$ has constant rank then $\rho_A(B)$ also has constant rank.
\end{itemize}
\end{prop}
\begin{proof}
\noindent\begin{itemize}
\item[$1)$] It is clear that $\rho_A(B)$ is a finitely generated distribution. Moreover, if $X,Y\in \Gamma(B)$ then, using that $[X,Y]_B=[X,Y]_A$, we deduce that
$$[\rho_A(X),\rho_B(Y)]=\rho_A[X,Y]_B$$
which implies that $\rho_A(B)$ is an involutive distribution. Thus, $\rho_A(B)$ is a generalized foliation.

On the other hand, if $x\in M,$ we have that
$$\dim (\rho_A(B)_x)\leq \dim B_x=\rank B.$$

\item[$2)$] $\mathcal{F}^B$ is a finitely generated distribution. In fact, let $U$ be an open subset of $M$  and $\{X_i\}$ be a basis of $\Gamma(\tau_B^{-1}(U)).$ Then, $\{X_i^c(a),X_i^v(a)\}$ is a generator system of $\mathcal{F}^B_a,$ for all $a\in \tau_A^{-1}(U).$

Moreover, since $B$ is a Lie subalgebroid of $A$, we deduce that $\mathcal{F}^B$ is an involutive distribution (see \eqref{eq:complete:vertical})).

Now, let $b$ be a point of $B_x$, with $x\in M$, and suppose that
$\{v_a,v_\beta\}$ is a basis of $B_x$, such that $\{\rho_A(v_a)\}$
(respectively, $\{v_\beta\}$) is a basis of $\rho_A(B_x)$
(respectively, $Ker (\rho_{A|B_x})).$ Then, we can choose an open
subset $U$ of $M$, with $x\in U,$ and a basis $\{X_a,X_\beta\}$ of
$\Gamma(\tau_{B}^{-1}(U))$ satisfying
$$X_a(x)=v_a,\;\;\; X_\beta(x)=v_\beta.$$
We complete the basis $\{X_a,X_\beta\}$ to a basis of $\Gamma(\tau_A^{-1}(U))$
$$\{X_a,X_\beta,X_{\bar{a}}\}$$
Next, we will assume, without the loss of generality, that on $U$ we have a system of local coordinates $(x^i)$. Thus, we can consider the corresponding local coordinates $(x^i,y^a,y^\beta,y^{\bar{a}})$ on $\tau_A^{-1}(U).$

Using \eqref{eq:complete:vertical:coordinate}, we deduce that
$$X_a^v(b)=\frac{\partial }{\partial y^a}_{|b},\;\;\; X_\beta^v(b)=\frac{\partial }{\partial y^\beta}_{|b}$$
$$(T\tau_B)(X_a^c(b))=\rho_A(v_a),\;\,\; X_\beta^c(b)\in V_b\tau_B=<\frac{\partial }{\partial y^a}_{|b}, \frac{\partial }{\partial y^\beta}_{|b}>$$ for all $a$ and $\beta.$

Therefore,
$$\dim \mathcal{F}^B_b=\dim (\rho_A(B)_{\tau_B(b)}) + \rank B$$
\end{itemize}
\end{proof}

\begin{rem}\label{r5.2'}
{\rm Note that if $a\in A_x$ then $(T_a\tau_A)(\mathcal{F}^B_a)=\rho_A(B_x).$ }
\end{rem}

Assume that $\rho_A(B)$ and $\mathcal{F}^B$ are regular foliations,
i.e., they have finite constant rank, $M/\rho_A(B)$ and $A/\mathcal{F}^B$
are differentiable manifolds, and
\[
\pi\colon M\to M/\rho_A(B)=\widetilde{M} \qquad\mbox{and}\qquad
\Pi\colon A\to A/\mathcal{F}^B=\widetilde{A}
\]
are submersions.

We define $\tau_{\widetilde{A}}\colon
\widetilde{A}=A/\mathcal{F}^B\to \widetilde{M}=M/\rho_A(B)$ such
that the following diagram is commutative
\[%%%COMMUTATIVE DIAGRAM
\xymatrix@C=4pc@R=4pc{ A \ar[r]^(.35){\Pi}\ar[d]^{\tau_A} & \widetilde{A}=A/\mathcal{F}^B\phantom{AA} \ar[d]^{\tau_{\widetilde{A}}}\\
                       M\ar[r]^(.35){\pi} & \widetilde{M}=M/\rho_A(B)}
\]
\vspace{3mm}

\noindent The map $\tau_{\widetilde{A}}$ is well defined. Indeed,
if $\Pi(a_x)=\Pi(a_{x'})\in A/\mathcal{F}^B$, with $a_x\in A_x$,
$a_{x'}\in A_{x'}$, $x,x'\in M$, then there exists a curve
$\sigma_A\colon [0,1] \to A$ continuous, piecewise differentiable,
tangent to $\mathcal{F}^B$, such that $\sigma_A(0)=a_x$ and
$\sigma_A(1)=a_{x'}$. Consider now the curve
$\sigma_M=\tau_A\smc\sigma_A\colon [0,1] \to M$ which results to
be continuous, piecewise differentiable, tangent to $\rho_A(B)$
(see Remark \ref{r5.2'}), such that
\[
\sigma_M(0)=x \qquad\mbox{and}\qquad \sigma_M(1)={x'}.
\]
Hence, $\pi(x)=\pi(x')$.

Note that $\tau_{\widetilde{A}}$ is a submersion since
$\tau_A,\Pi$ and $\pi$ are submersions. On the other hand, we have a
vector bundle $\tau_{\overline{A}}\colon \overline{A}=A/B \to M$
such that the following diagram is commutative
\[%%%COMMUTATIVE DIAGRAM
\xymatrix@C=3pc@R=3pc{ A \ar[r]^(.35){\overline{\Pi}}\ar[dr]^{\tau_A} & A/B=\overline{A}\phantom{AA} \ar[d]^{\tau_{\overline{A}}}\\
                       & M}
\]
\vspace{3mm}
In fact, $\overline{\Pi}$ is a vector bundle epimorphism.
\noindent Therefore, we can induce a smooth map $\widetilde{\Pi}\colon A/B \to
A/\mathcal{F}^B$ making commutative the following diagram
\[%%%COMMUTATIVE DIAGRAM
\xymatrix@C=4pc@R=4pc{ \overline{A}=A/B \ar[r]^(.40){\widetilde{\Pi}}\ar[d]^{\tau_{\overline{A}}} & \widetilde{A}=A/\mathcal{F}^B\phantom{AAA} \ar[d]^{\tau_{\widetilde{A}}}\\
                       M\ar[r]^(.40){\pi} & \widetilde{M}=M/\rho_A(B)}
\]
\vspace{3mm}

\noindent Indeed, if $a'_x - a_x \in B_x$, then we can consider
the curve $\sigma\colon [0,1] \to A_x $ defined by
\[
\sigma(t)=a_x+t (a'_x-a_x).
\]
Note that $\sigma(0)=a_x$ and $\sigma(1)=a_{x'}$. Moreover,
$\dot{\sigma}(t_0)=(a'_x-a_x)^v_{\sigma(t_0)}\in\mathcal{F}^B_{\sigma(t_0)}$.
Hence, $\Pi(a_x)=\Pi(a'_x)$. $\widetilde{\Pi}$ is a smooth map since $\overline{\Pi}:A\to \bar{A}=A/B$ is a submersion.

In order to guarantee that $\tau_{\widetilde{A}}$ is a vector bundle, we suppose that $B$  satisfies {\it the condition $\mathcal{F}^B$ }, i.e.
\[
a_x,a'_x\in A_x \mbox{ are in the same leaf of } \mathcal{F}^B\qquad
\Longleftrightarrow  \qquad a'_x-a_x\in B_x.
\]
for any $x\in M.$ Note that if $a'_x-a_x\in B_x$ then $a_x$ and $a'_x$ are in the same leaf of  $\mathcal{F}^B.$ In order to prove that the  condition $\mathcal{F}^B$ holds it is only necessary to verify the other implication.

\begin{prop}\label{prop:reduced:is:vector:bundle}
Assume that  $\rho_A(B)$ and $\mathcal{F}^B$ are regular foliations
and that $B$ satisfies the condition $\mathcal{F}^B$. Then,
$\tau_{\widetilde{A}}\colon \widetilde{A}=A/\mathcal{F}^B\to
\widetilde{M}=M/\rho_A(B)$ is a vector bundle, the fiber of $\widetilde{A}$ over the point $\pi(x)\in \widetilde{M}$ is
isomorphic to the quotient vector space $A_x/B_x$  and the diagram
\[%%%COMMUTATIVE DIAGRAM
\xymatrix@C=3pc@R=3pc{ \overline{A}=A/B \ar[r]^(.40){\widetilde{\Pi}}\ar[d]^{\tau_{\overline{A}}} & \widetilde{A}=A/\mathcal{F}^B\phantom{AAA} \ar[d]^{\tau_{\widetilde{A}}}\\
                       M\ar[r]^(.40){\pi} & \widetilde{M}=M/\rho_A(B)}
\]
\vspace{3mm}
\noindent is a vector bundle epimorphism. In fact, the restriction of $\widetilde{\Pi}$ to the fiber $\bar{A}_x=A_x/B_x$ is a linear isomorphism on $\widetilde{A}_{\pi(x)}.$
\end{prop}
\begin{proof}
We apply Lemma \ref{induced:VB} to the following diagram
\[%%%COMMUTATIVE DIAGRAM
\xymatrix@C=3pc@R=3pc{ \overline{A}=A/B \ar[d]^{\widetilde{\Pi}} \ar[r]^{\tau_{\overline{A}}} &  M \ar[d]^{\pi}\\
                       \widetilde{A}=A/\mathcal{F}^B \ar[r]^(.40){\tau_{\widetilde{A}}} & \widetilde{M}=M/\rho_A(B)}
\]
\vspace{3mm}

\noindent Note that $\pi$ and $\tau_{\widetilde{A}}$ are
submersions. Then, for all $x\in M$,
$\tau_{\widetilde{A}}^{-1}(\pi(x))$ is a regular submanifold of
$\widetilde{A}$ and $T_{\Pi(a_x)}
\tau_{\widetilde{A}}^{-1}(\pi(x))=\ker T_{\Pi(a_x)}
\tau_{\widetilde{A}}$ for all $a_x\in A_x$.

We will see that
\[
\widetilde{\Pi}_x\colon A_x/B_x \to
\tau_{\widetilde{A}}^{-1}(\pi(x))
\]
is a surjective submersion.

Indeed, let
$\Pi(a_{x'})\in\tau_{\widetilde{A}}^{-1}(\pi(x))$. Then,
$\tau_{\widetilde{A}}(\Pi(a_{x'}))=\pi(x)$. Hence,
$\pi(x)=\pi(x')$. Therefore, there exists a continuous, piecewise
differentiable path $\sigma\colon [0,1] \to M$ tangent to
$\rho_A(B)$ such that $\sigma(0)=x$ and $\sigma(1)=x'$. In each
differentiable piece we can find $X\in\Gamma(B)$ such that
$\sigma$ is an integral curve of $\rho_A(X)$. Assume, without the loss of generality, that the curve $\sigma$ is smooth and let $\psi\colon
\Rr\times M\to M$ be the flow of $\rho_A(X)$. Then, $\psi_x(0)=x$,
\[
\frac{\d\psi_x}{\d t}=\rho_A(X)(\psi_x(t))
\]
and there exists $t_0\in\Rr$ such that $\psi_{t_0}(x)=x'$. Let
$\varphi\colon \Rr\times A\to A$ be the flow of $X^c\in\X(A)$.
Since $X^c$ projects on $\rho_A(X)$ we have that the following
diagram is commutative for any $t$
\[%%%COMMUTATIVE DIAGRAM
\xymatrix@C=3pc@R=3pc{ A \ar[d]^{\tau_A} \ar[r]^{\varphi_t} &  A \ar[d]^{\tau_A}\\
                       M \ar[r]^{\psi_t} & M}
\]
\vspace{3mm}

\noindent Hence $\varphi_{-t_0}(a_{x'})\in A_x$ and hence we have
a curve  $\varphi_{a_{x'}}$ on $A$ such that
\[
\frac{\d \varphi_{a_{x'}}}{\d t}(t)=X^c(\varphi_{a_{x'}}(t))\in
\mathcal{F}^B_{\varphi_{a_{x'}}(t)},
\]
$\varphi_{a_{x'}}(0)=a_{x'}$ and
$\varphi_{a_{x'}}(-t_0)=\varphi_{-t_0}(a_{x'})$. Thus,
\[
\widetilde{\Pi}_x\overline{\Pi}_x(\varphi_{-t_0}(a_{x'}))=\Pi_{x'}(\varphi_{-t_0}(a_{x'}))=\Pi_{x'}(a_{x'}),
\]
where $\Pi_x$ and $\Pi_{x'}$ (respectively, $\overline{\Pi}_x$) are  the
restrictions of $\Pi$ (respectively, $\overline{\Pi}$) to the fiber over $x$ and $x'$.
So, $\widetilde{\Pi}_x$ is surjective. Moreover, using that the following diagram
\[%%%COMMUTATIVE DIAGRAM
\xymatrix@C=3pc@R=3pc{ A_x/B_x \ar[r]^(.40){\widetilde{\Pi}_x} & \tau_{\widetilde{A}}^{-1}(\pi(x))\phantom{AA}\\
                        A_x  \ar[u]_{\overline{\Pi}_x} \ar[ur]_{\Pi_x}}
\]
\vspace{3mm}
is commutative, we deduce that $\widetilde{\Pi}_x:A_x/B_x\to \tau_{\widetilde{A}}^{-1}(\pi(x))$ is smooth.

As a matter of fact $\widetilde{\Pi}_x$ is a submersion, i.e.,
\[
T_{\overline{\Pi}(a_x)}\widetilde{\Pi}_x\colon
T_{\overline{\Pi}(a_x)}(A_x/B_x)\to
T_{\Pi(a_x)}(\tau_{\widetilde{A}}^{-1}(\pi(x)))
\]
is surjective. Indeed, let $\widetilde{X}\in
T_{\Pi(a_x)}(\tau_{\widetilde{A}}^{-1}(\pi(x)))=\ker
T_{\Pi(a_x)}\tau_{\widetilde{A}}$. Then, since $\Pi\colon A\to
\widetilde{A}=A/\mathcal{F}^B$ is a submersion, there exists $X\in
T_{a_x}A$ such that
\begin{equation}\label{eq:X:tilde}
\widetilde{X}=T_{a_x}\Pi(X).
\end{equation}
Hence,
\begin{align*}
0=T_{\Pi(a_x)}\tau_{\widetilde{A}} (\widetilde{X})&=
T_{\Pi(a_x)}(\tau_{\widetilde{A}} \smc \Pi)(X)=T_{a_x} (\pi \smc
\tau_A)(X),
\end{align*}
i.e. $T_{a_x} \tau_A(X)\in\ker T_x\pi=V_x\pi=\rho_A(B_x)$.

From Remark  \ref{r5.2'}, we deduce that there exists $Y\in \mathcal{F}^B_{a_x}$ such that
\[
T_{a_x}\tau_A (Y)=T_{a_x}\tau_A (X),
\]
or equivalently $X-Y\in \ker T_{a_x} \tau_A = V_{a_x}(\tau_A)=
T_{a_x}A_x$.

Consider now $\overline{\Pi}_x\colon A_x\to A_x/B_x$. Then,
\[
W_x=T_{a_x} \overline{\Pi}_x (X-Y)\in T_{\overline{\Pi}({a_x})}
(A_x/B_x).
\]
We will see now that
$T_{\overline{\Pi}({a_x})}\widetilde{\Pi}_x(W_x)=\widetilde{X}$.
In fact,
$$
T_{\overline{\Pi}({a_x})}\widetilde{\Pi}_x(W_x)=T_{a_x}(\widetilde{\Pi}_x \smc
\overline{\Pi}_x)(X-Y)=T_{a_x}\Pi_x(X-Y).
$$

 On the other hand, from \eqref{eq:X:tilde} and since
$Y\in\mathcal{F}^B_{a_x}$,
\[
\widetilde{X}=T_{a_x}\Pi(X)=T_{a_x}\Pi(X-Y).
\]

Therefore, $T_{\bar\Pi(a_x)}\widetilde{\Pi}_x$ is surjective. Indeed, $T_{\bar\Pi(a_x)}\widetilde{\Pi}_x$ is a linear isomorphism since
\[
\dim T_{\overline{\Pi}(a_x)}(A_x/B_x)=\dim A_x-\dim B_x
\]
and by using Proposition \ref{prop:gen:foliation}
\begin{align*}
\dim T_{\widetilde{\Pi}(a_x)}\tau_{\widetilde{A}}^{-1}(\pi(x))&= \dim \widetilde{A} - \dim \widetilde{M}\\
&=\dim A - \rank \mathcal{F}^B- \dim \widetilde{M}\\
&=\dim A_x -\dim B_x.
\end{align*}
Thus,
\[
\widetilde{\Pi}_x\colon A_x/B_x \to \tau_{\widetilde{A}}^{-1}(\pi(x))
\]
is a local diffeomorphism. Therefore (using that $\widetilde{\Pi}_x$ is bijective), $\widetilde{\Pi}_x$ is  a global diffeomorphism.

Finally, if $\pi(x)=\pi(x')$, it is clear that
\[
\widetilde{\Pi}_{x'}^{-1} \smc \widetilde{\Pi}_x\colon A_x/B_x \to A_{x'}/B_{x'}
\]
is  a linear  isomorphism.
\end{proof}

\begin{prop}\label{prop:quotient:alg}
Under the same conditions as in Proposition \ref{prop:reduced:is:vector:bundle},
we can define a Lie algebroid structure
on the vector bundle
\[
\tau_{\widetilde{A}}\colon \widetilde{A}=A/\mathcal{F}^B\to
\widetilde{M}=M/\rho_A(B)
\]
such that the diagram
\[%%%COMMUTATIVE DIAGRAM
\xymatrix@C=3pc@R=3pc{ A \ar[r]^(.35){\Pi}\ar[d]_{\tau_A} & \widetilde{A}=A/\mathcal{F}^B\phantom{AAA} \ar[d]^{\tau_{\widetilde{A}}}\\
                        M\ar[r]^(.35){\pi} & \widetilde{M}=M/\rho_A(B)}
\]
\vspace{3mm}

\noindent is an epimorphism of Lie algebroids.
\end{prop}
\begin{proof}
Due to Proposition \ref{prop:epi},  it is enough to prove the following facts
\begin{itemize}
\item[$i)$] If $ X,Y\in\Gamma_p^{\Pi}(A)$ then $\brr{X,Y}_A\in\Gamma_p^{\Pi}(A)$ and
\item[$ii)$] If $X\in\Gamma_p^{\Pi}(A)$ and
$Y\in\Gamma(\ker\Pi)$ then
$\brr{X,Y}_A\in\Gamma(\ker\Pi).$
\end{itemize}
Here $\Gamma_p^{\Pi}(A)$ is the space of $\Pi$-projectable
sections.

Note firstly that $\ker\Pi=B$ (see Proposition \ref{prop:reduced:is:vector:bundle} ). Then,   we will prove that
\[
\Gamma_p^{\Pi}(A)=\{X\in \Gamma(A) \;|\; \brr{X,Y}\in\Gamma(B),\;\;\forall Y\in\Gamma(B) \}.
\]
Once we prove that, condition $i)$ above follows by the Jacobi
identity and $ii)$ is a direct consequence.

 Let $X\in\Gamma_p^{\Pi}(A)$ and
$Y\in\Gamma(B)$. We denote by  $\psi\colon\Rr\times M\to M$ the flow
of $\rho_A(Y)$ and by $\varphi\colon\Rr\times A\to A$ the flow of
$Y^c\in \X(A)$. Using that $Y^c$ is $\tau_A$-projectable over $\rho_A(Y)$, we deduce that the following diagram
\[%%%COMMUTATIVE DIAGRAM
\xymatrix@C=3pc@R=3pc{ A \ar[r]^{\varphi_t}\ar[d]_{\tau_A} & A \ar[d]^{\tau_A}\\
                        M\ar[r]^{\psi_t} & M}
\]
\vspace{3mm} is commutative and that the couple $(\varphi_t,\psi_t)$ is a Lie algebroid morphism (see \cite{MX2}).
Note that
\begin{equation}\label{eq:pigrande}
\Pi \smc \varphi_t(a_x)=\Pi \smc \varphi_{a_x}(t)=\Pi(a_x)
\quad\mbox{ and }\quad
\pi \smc \psi_t(x)=\pi\smc\psi_x(t)=\pi(x).
\end{equation}
On the other hand, $X$ is $\Pi$-projectable, thus there exists
$\widetilde{X}\in\Gamma(\widetilde{A})$ such that
\[
\widetilde{X}\smc \pi=\Pi\smc X.
\]
Therefore, by using \eqref{eq:pigrande}
\begin{align*}
\Pi(X(\psi_t(x))-\varphi_t(X(x)))=\widetilde{X}(\pi(\psi_t(x)))-
\Pi(X(x))=0.
\end{align*}
In consequence, there exists $Z_t\in\Gamma(B)$ such that
\[
X(\psi_t(x))-\varphi_t(X(x))=Z_t(\psi_t(x)),
\]
i.e.,
\[
(X-Z_t)(\psi_t(x))=\varphi_t(X(x)).
\]
Thus, if  $\varphi_t^\ast\colon A^\ast\to A^\ast$ is the dual map of
$\varphi_t\colon A\to A$, it follows that
\[
\widehat{X}-\widehat{Z_t}= \widehat{X}\smc \varphi_t^\ast
\]
By derivation and using \eqref{eq:completo:star:flusso} we obtain
that
\[
Y^{\ast c}(\widehat{X})=\frac{\d}{\d t}_{|t=0}(\widehat{X}\smc
\varphi_t^\ast)= -\frac{\d}{\d t}_{|t=0}\widehat{Z_t}.
\]
We denote by $Z$ the section of $B$ characterized by
$\widehat{Z}=\frac{\d}{\d t}_{|t=0}\widehat{Z_t}$. By using
\eqref{eq:completo:star}  we deduce
\[
\widehat{[X,Y]}=\widehat{Z}
\]
so that $[X,Y]\in \Gamma(B)$.

Conversely, let $X\in\Gamma(A)$ such that for all $Y\in\Gamma(B)$,
\[
[X,Y]\in \Gamma(B).
\]
We will see that $X\in\Gamma^{\Pi}_p(A)$. In order to prove it, we introduce the map
\[
\widetilde{X}\colon \widetilde{M}=M/\rho_A(B) \to \widetilde{A}=A/\mathcal{F}^B
\]
given  by $\widetilde{X}(\pi(x))=\Pi(X(x))$, which is well
defined.

 In fact, suppose that $x,x'\in M$ with $\pi(x)=\pi(x')$. Then
there exists a map $\sigma\colon [0,1]\to M$ continuous, piecewise
differentiable, tangent to $\rho_A(B)$ such that $\sigma(0)=x$ and
$\sigma(1)=x'$. So, in each piece there exists $Y\in\Gamma(B)$
such that $\sigma$ is the integral curve of $\rho_A(Y)$. Assume, without the loss of generality, that $\sigma$ is smooth and denote by $\psi_t:{\Bbb R}\times M\to M$ the flow of the vector field $\rho_A(Y)$. We have that there exists $t_0\in {\Bbb R}$ such that
\[
\psi_{t_0}(x)=x'.
\]
Now, let $\varphi:{\Bbb R}\times A\to A$ be the flow of the vector field $Y^c$. Then, for each $t\in {\Bbb R}$, the following diagram
\[%%%COMMUTATIVE DIAGRAM
\xymatrix@C=3pc@R=3pc{ A \ar[r]^{\varphi_t}\ar[d]_{\tau_A} & A \ar[d]^{\tau_A}\\
                        M\ar[r]^{\psi_t} & M}
\]
is commutative.

On the other hand, using that $Y\in \Gamma(B),$ we deduce that there exists $Z\in \Gamma(B)$ such that
\[
[X,Y]=Z.
\]
This implies the corresponding relation between the linear maps
\[
(Y^\ast)^c(\widehat{X})=\widehat{[Y,X]}=-\widehat{Z}
\]
or equivalently
\[
\frac{\d}{\d t}_{|t=0}(\widehat{X}\smc \varphi_t^\ast)=
-\widehat{Z}.
\]
So, for each $t_1$
\begin{equation}\label{eq:zt1}
\frac{\d}{\d t}_{|t=t_1}(\widehat{X}\smc \varphi_t^\ast)=
\widehat{Z}_{t_1},
\end{equation}
where we have denoted by $Z_{t_1}$ the section of the vector bundle $\tau_B:B\to M$ which is characterized by $\widehat{Z}_{t_1}=-\widehat{Z}\smc
\varphi^\ast_{t_1}$. Since
$\widehat{X}\smc\varphi_0^\ast=\widehat{X}$, by integrating
\eqref{eq:zt1} we have
\[
\widehat{X}\smc\varphi_{t_1}^\ast=\widehat{X}+\widehat{W}_{t_1}
\]
for each $t_1\in \Rr$ with $W_{t_1}\in\Gamma(B).$  Hence,
we get the relation
\[
\varphi_{t}\smc X- X \smc \psi_t=W_{t}\smc \psi_t
\]
By applying $\Pi$, we get
\[
\Pi\smc \varphi_t\smc X= \Pi \smc X \smc \psi_t
\]
Now, since the vector field $Y^c$ on $A$ is tangent to the foliation $\mathcal{F}^B$, it follows that $\Pi\smc \varphi_t=\Pi.$ Therefore,
$$\Pi\smc X=\Pi\smc X\smc \psi_t$$
which implies that
\[
\Pi\smc X(x)= \Pi \smc X \smc \psi_{t_0}(x)=\Pi\smc X(x').
\]
In conclusion, $\widetilde{X}$ is well defined and  $X$ is
$\Pi$-projectable.
\end{proof}

%%%%%%%%%%%%%%%%%%%%%%%%%%%%%%%%%%%%%%%%%%%%%%%%%%%%%%%%%%%%%%%%%%%%%%%
%%%%%%%%%%%%%%%%%%%%%%%%%%%%%%%%%%%%%%%%%%%%%%%%%%%%%%%%%%%%%%%%%%%%%%%
\section{The Reduced nondegenerate symplectic-Nijenhuis Lie algebroid}   %
\label{sec:reduced:nondegenerate}%%%%%%%%%%%%%%%%%%%%%%%%%%%%%%%%%%%%%%
%%%%%%%%%%%%%%%%%%%%%%%%%%%%%%%%%%%%%%%%%%%%%%%%%%%%%%%%%%%%%%%%%%%%%%%

First of all, we will prove a result which will be useful in the sequel

\begin{prop}\label{6.0'} Let $(A,P,N)$ be a Poisson-Nijenhuis Lie algebroid. If $l$ is a positive integer then the couple $(P,N^l)$ is a Poisson-Nijenhuis structure on the Lie algebroid $A.$
\end{prop}
\begin{proof}
It is well-known that $N^l$ is a Nijenhuis operator (see, for
instance, Lemma 1.2 in \cite{KM}).

On the other hand, it is clear that
\begin{equation}\label{(1)}
P^\sharp \smc N^{*l}=N^l\smc P^\sharp.
\end{equation}
In addition, a long straightforward computation, using \eqref{(1)}, proves that
$$C(P,N^l)(\beta,\beta')=0, \mbox{ for }\beta,\beta'\in \Gamma(A^*)$$
if and only if
\begin{equation}\label{(2)}
(\lie^A_{P^\sharp(\beta)}N^l)(X)=P^\sharp (\lie_X^AN^{*l}\beta-\lie_{N^lX}^A\beta),\;\;\; \mbox{ for }\beta\in\Gamma(A^*) \mbox{ and } X\in \Gamma(A).
\end{equation}

Thus, we must prove \eqref{(2)}.

We will proceed by induction on $l$. Note that
$$(\lie^A_{P^\sharp(\beta)}N^l)(X)=(\lie_{P^\sharp(\beta)}^AN^{l-1})(NX)  + N^{l-1}[{P^\sharp(\beta)},NX]_A-N^{l}[{P^\sharp(\beta)},X]_A$$

Therefore,
$$
\begin{array}{rcl}(\lie_{P^\sharp (\beta)}^AN^l)(X)&=&P^\sharp (\lie_{NX}^AN^{*l-1}\beta)-P^\sharp(\lie_{N^lX}^A\beta) + N^{l-1}[P^\sharp(\beta), NX]_A\\&&-N^l[P^\sharp \beta,X]_A\end{array}$$

Now, since
$$P^\sharp (\lie_{NX}^A\gamma)=P^\sharp(\lie_X^AN^*\gamma)-(\lie_{P^\sharp\gamma}^AN)(X),\;\;\; \mbox{ for }\gamma\in \Gamma(A^*),$$
we deduce that
$$\begin{array}{rcl}
(\lie_{P^\sharp(\beta)}^AN^l)(X)&=&P^\sharp (\lie_X^AN^{*l}\beta-\lie_{N^lX}^A\beta)-(\lie^A_{P^\sharp(N^{*l-1}\beta)},N)(X)
\\[5pt]&&+ N^{l-1}[P^\sharp\beta,NX]_A-N^l[P^\sharp\beta,X]_A\end{array}$$
which implies that
$$\begin{array}{rcl}
(\lie^A_{P^\sharp(\beta)}N^l)(X)&=&P^\sharp(\lie_X^AN^{*l}\beta-\lie_{N^lX}^A\beta)-
[P^\sharp(N^{*l-1}\beta),NX]_A \\[5pt]&&+ N[P^\sharp(N^{*l-1}\beta),X]_A + N^{l-1}[P^\sharp(\beta),NX]_A\\[5pt]&&-N^l[P^\sharp \beta,X]_A\\[5pt] &=&P^\sharp(\lie_X^AN^{*l}\beta-\lie_{N^{l}X}^A\beta)-[N^{l-1}(P^\sharp\beta),NX]_A \\[5pt]&&+ N[N^{l-1}(P^\sharp \beta),X]_A + N^{l-1}[P^\sharp \beta,NX]_A -N^l[P^\sharp \beta,X]_A
\end{array}$$

On the other hand, using that $0=\T_N(N^{l-r}(P^\sharp\beta),X)$ for $2\leq r   \leq l$, it follows that
$$-[N^{l-1}(P^\sharp \beta),NX]_A + N[N^{l-1}(P^\sharp \beta),X]_A + N^{l-1}[P^\sharp\beta,NX]_A- N^l[P^\sharp \beta,X]_A=0.$$

This ends the proof of the result.
\end{proof}

 Let $(A, P, N)$ be a Poisson-Nijenhuis Lie algebroid.
Consider now for any fixed $x\in M$ the endomorphism $N_x\colon
A_x\to A_x$. Recall \cite{HEU,KM} that there exists a smallest
positive integer $k$ such that the sequences of nested subspaces
\[
\im N_x \supseteq \im N^2_x\supseteq\ldots
\]
and
\[
\ker N_x \subseteq \ker N^2_x\subseteq\ldots
\]
both stabilize at rank $k$. That is,
\[
\im N_x^k=\im N_x^{k+1}=\ldots,
\]
and
\[
\ker N_x^k=\ker N_x^{k+1}=\ldots. \]
The integer $k$ is called the \emph{Riesz index} of $N$ at $x$.
\begin{lem}\label{6.0} If the Riesz index of  $N$ at $x$ is $k$ then
$$A_x=ImN_x^k\oplus Ker N_x^k$$
\end{lem}
\begin{proof}
It is clear that
$$\dim(ImN_x^k) + \dim (\ker N_x^k)=\dim A_x$$
\end{proof}

Next, we will prove the following result

\begin{prop}\label{prop:reduction:PN:alg}
Let $(A, P, N)$ be a Poisson-Nijenhuis Lie algebroid with constant
Riesz index $k$ and such that the dimension of the subspace $\ker N_x^k$ (respectively, $\im N_x^k$) is constant, for all $x\in M$. Then:
\begin{itemize}
\item[$i)$] The dimension of the subspace $\im N_x^k$ (respectively, $\ker N_x^k$) is constant, for all $x\in M.$
\item[$ii)$] The vector subbundles $\ker N^k$ and $\im N^k$ are Lie subalgebroids of $A.$
\end{itemize}
\end{prop}
\begin{proof}
$i)$ it follows from Lemma \ref{6.0}.
\medskip

Since $N^k$  is a Nijenhuis operator, we have that
\begin{equation}\label{Lehmann}
[N^k X,N^k Y]_A=N^k [X,Y]_{N^k},  \qquad\mbox{for any } X,Y\in
\Gamma(A),
\end{equation}
where $\br_{N^k}$ is the bracket defined as in
\eqref{eq:deformed:bracket}. Thus, $\im N^k$ is a Lie subalgebroid of $A$.

Now, suppose that $X,Y\in \Gamma(A)$ are sections of $A$ satisfying  $N^k
X=N^k Y=0$. Then, using  \eqref{eq:deformed:bracket} and \eqref{Lehmann}, we deduce that
\[
0=[N^kX,N^kY]_A=N^k [X,Y]_{N^k}=-N^{2k} [X,Y]_A.
\]
Hence, $[X,Y]_A\in\Gamma(\ker N^{2k})=\Gamma(\ker N^k)$. This implies that $\ker N^k$ is a Lie subalgebroid of $A$.

\end{proof}

Let $(A, P, N)$ be a Poisson-Nijenhuis Lie algebroid with constant
Riesz index $k$. Suppose that the dimension of the subspace $\ker N_x^k$ is constant, for all $x\in M.$ Then, we may consider
the Lie subalgebroid $\ker N^k$ of $A$ and the corresponding generalized foliations $\rho_A(\ker N^k)$ on $M$ and
${\mathcal F}^{\ker N^k}$ on $A$.

As in Section \ref{section5}, we will assume that these foliations are regular and that the condition ${\mathcal F}^{\ker N^k}$ holds, that is, if $a_x,a_x'\in A_x$ we have that
$$a'_x-a_x\in \ker N_x^k\Leftrightarrow a'_x \mbox{ and } a_x \mbox{ belong to the same leaf of the foliation ${\mathcal F}^{\ker N^k}.$}$$

Under these conditions, the space $\widetilde{A}=A/{\mathcal F}^{\ker N^k}$ of the leaves of the foliation ${\mathcal F}^{\ker N^k}$ is a Lie algebroid over the quotient manifold $\widetilde{M}=M/\rho_A({\ker N^k})$ and the canonical projections $\Pi:A\to \widetilde{A}= A/{\mathcal F}^{\ker N^k}$ and  $\pi:M\to \widetilde{M}=M/\rho_A(\ker N^k)$ define a Lie algebroid epimorphism

\[%%%COMMUTATIVE DIAGRAM
\xymatrix@C=4pc@R=4pc{ A \ar[r]^(.35){\Pi}\ar[d]^{\tau_A} & \widetilde{A}=A/\mathcal{F}^{\ker N^k}\phantom{AA} \ar[d]^{\tau_{\widetilde{A}}}\\
                       M\ar[r]^(.35){\pi} & \widetilde{M}=M/\rho_A({\ker N^k})}
\]
Note that $\ker \Pi=\ker N^k$ and thus,
$$V\pi=\rho_A(\ker N^k)=\rho_A(\ker \Pi).$$

Next, we will prove that $P$ and $N$ are $\Pi$-projectable. Indeed, we have that
\[
N(\Gamma(\ker N^k))\subseteq \Gamma(\ker N^k).
\]
Moreover, if $\xi\in \Gamma(\ker N^k)$ one may see that $\lie^A_{\xi}N(\Gamma_p(A))\subseteq \Gamma(\ker N^k)$.
In order to prove this relation, we recall that
\[
\Gamma_p(A)=\{X\in \Gamma(A) \;|\; [X,\xi]_A\in\Gamma(\ker
N^k)\quad \forall\xi\in \Gamma(\ker N^k)\}.
\]
Now,  if $X\in \Gamma_p(A)$ and $\xi\in\Gamma(\ker N^k)$,  we get
\[
N^k(\lie^A_{\xi}N(X))=N^k([\xi,NX]_A-N[\xi,X]_A)=N^k[\xi,NX]_A.
\]
By using the fact that $N$ has zero torsion, it follows that
\begin{equation}\label{53'}
0=N^{k-1}(\T_N(N^{k-1}\xi,X))=-N^k[N^{k-1}\xi, NX]_A
\end{equation}

Thus, from (\ref{53'}), we deduce that
$$0=N^k(\T_N(N^{k-2}\xi,X))=-N^{k+1}[N^{k-2}\xi,NX]_A$$
and, since $\ker N^{k+1}=\ker N^k,$ we obtain that
$$N^k[N^{k-2}\xi, NX]_A=0$$
Proceeding in a similar way, we may prove that
$$N^k[N^{k-3}\xi, NX]_A=N^k[N^{k-4}\xi,NX]=\dots =N^k[\xi,NX]=0.$$
Therefore,
$\lie^A_{\xi}N(X)\in \Gamma(\ker N^k)$ and $N$ is $\Pi$-projectable (see Proposition \ref{prop:projectable:N}).

 To see that $P$ is $\Pi$-projectable, we have to prove that (see Proposition \ref{prop:projectable:bivector})
\begin{equation}\label{eq:lieP}
[{\xi},P]_A^\sharp(\Gamma_p(A^\ast))\subseteq \Gamma(\ker N^k)
\qquad \forall\xi\in \Gamma(\ker N^k).
\end{equation}
From Proposition \ref{prop:projectable:sections},
\[
\Gamma_p(A^\ast)=\{\al\in \Gamma(A^\ast) \;|\;
\lie^A_{\xi}\al=0,\quad\al(\xi)=0,\quad \forall\xi\in \Gamma(\ker
N^k)\}.
\]
If  $\al\in \Gamma_p(A^\ast)$ then
\begin{equation*}
\begin{split}
N^k(([{\xi}, P]_A)^\sharp(\al))
&=N^k(i_{\al}[{\xi},P]_A)=N^k([{\xi},i_{\al}P]_A-
i_{\lie^A_{\xi}\al}P)\\ &=N^k[\xi,P^\sharp \al] _A
=-[{P^\sharp \al},N^k\xi]_A+(\lie^A_{P^\sharp
\al}N^k)(\xi)=(\lie^A_{P^\sharp \al}N^k)(\xi).
\end{split}
\end{equation*}
Hence, using Proposition \ref{6.0'}, we deduce that
$$N^k([\xi,P]_A^\sharp (\alpha))=P^\sharp (\lie_\xi^AN^{*k}\alpha).$$

On the other hand, since $\al\in \Gamma_p(A^\ast)$, we get
\begin{equation*}
\begin{split}
\lie^A_{\xi} (N^{\ast k} \al)(X)&=\rho_A(\xi)(\al(N^k X))-
\al(N^k[\xi,X]_A)\\&=\lie^A_{\xi}
\al(N^k X)+ \al([\xi,N^k X]_A-N^k[\xi,X]_A)
\\&=\al([\xi,N^k X]_A-N^k[\xi,X]_A).
\end{split}
\end{equation*}
Moreover, since the torsion of $N^k$ is zero, we have that
\[
0=\T_{N^k}(\xi,X)=-N^k[\xi,N^kX]_A+N^{2k}[\xi,X]_A,
\]
that is,
$$[\xi,N^kX]_A-N^k[\xi,X]_A\in \Gamma(\ker N^k).$$
This implies that
$$\alpha([\xi,N^kX]_A-N^k[\xi,X]_A)=0$$
Hence,
$$N^k([\xi, P]_A^\sharp (\alpha))=P^\sharp (\lie_\xi^AN^{*k}\alpha)=0$$
and \eqref{eq:lieP} holds.

 Therefore, using Theorem  \ref{thm:PN:epi}, we have the following result.

\begin{thm}
Let $(A,\br_A,\rho_A,P,N)$ be a Poisson-Nijenhuis Lie
algebroid such that
\begin{itemize}
\item[$i)$] $N$ has constant Riesz index $k$;
\item[$ii)$] The dimension of the subspace $\ker N_x^k$ is constant, for all $x\in M$ (thus, $B=\ker N^k$ is a vector subbundle of $A$) and
\item[$iii)$] $\rho_A(B)$ and $\mathcal{F}^B$ are regular
foliations and the condition $\mathcal{F}^B$ is satisfied for
$B=\ker N^k$.
\end{itemize}
Then, we may induce a Poisson-Nijenhuis Lie algebroid structure
\linebreak
$(\br_{\widetilde{A}},\rho_{\widetilde{A}},\widetilde{P},\widetilde{N})$
on $\widetilde{A}=A/\mathcal{F}^B$ such that  $\Pi:A\to \widetilde{A}=A/\mathcal{F}^B$  is a Poisson-Nijenhuis Lie algebroid epimorphism over $\pi:M\to \widetilde{M}=M/\rho_A(B).$
\end{thm}

In the particular case of symplectic-Nijenhuis Lie algebroids, we may prove the following result
\begin{thm}\label{thm:nondegenerate:N}
Let $(A,\br_A,\rho_A, \Omega,N)$ be a symplectic-Nijenhuis Lie
algebroid on the manifold $M$ such that
\begin{itemize}
\item[$i)$] $N$ has constant Riesz index $k$;
\item[$ii)$] The dimension of the subspace $\ker N_x^k$ is constant, for all $x\in M$ (thus, $B=\ker N^k$ is a vector subbundle of $A$) and
\item[$iii)$]$\rho_A(B)$ and $\mathcal{F}^B$ are regular
foliations and the condition $\mathcal{F}^B$ is satisfied for
$B=\ker N^k$.
\end{itemize}
Then, we may induce a symplectic-Nijenhuis Lie algebroid structure
on $\widetilde{A}$ with a nondegenerate  Nijenhuis tensor,
such that the couple $\Pi:A\to \widetilde{A}=A/\mathcal{F}^B$  and $\pi:M\to \tilde{M}=M/\rho_A(B)$ is a Poisson-Nijenhuis Lie algebroid epimorphism.
%(i.e. $\Omega_P$ is symplectic and $\widetilde{N}$ is an automorphism).
\end{thm}
\begin{proof} Denote by $(\widetilde{P},\widetilde{N})$ the Poisson-Nijenhuis structure which is defined in the previous theorem.
It remains to prove that $\widetilde{P}$ and $\widetilde{N}$ are
nondegenerate.

Firstly, we show that $\widetilde{N}$ is nondegenerate, i.e. that $\widetilde{N}_{\pi(x)}\colon
\tau_{\widetilde{A}}^{-1}(\pi(x)) \to
\tau_{\widetilde{A}}^{-1}(\pi(x))$ is an isomorphism, for all $x\in M$. Consider the
following diagram
\[%%%COMMUTATIVE DIAGRAM
\xymatrix@C=3pc@R=3pc{ A_x \ar[r]^{N_x} \ar[d]_{\overline{\Pi}_x}
                        \ar@/_3pc/[dd]_{\Pi_x}
                        & A_x \ar[d]^{\overline{\Pi}_x}\ar@/^3pc/[dd]^{\Pi_x}\\
                        A_x/\ker N^k_x \ar[r]^{\overline{N}_x}\ar[d]_{\widetilde{\Pi}_x}
                        & A_x/\ker N^k_x \ar[d]^{\widetilde{\Pi}_x}\\
                        \tau_{\widetilde{A}}^{-1}(\pi(x))\ar[r]^{\widetilde{N}_{\pi(x)}}
                        &\tau_{\widetilde{A}}^{-1}(\pi(x))}
\]
\vspace{3mm}

\noindent where $\overline{\Pi}_x$ and $\widetilde{\Pi}_x$ are
defined as in Section \ref{section5}.
Assume that $\widetilde{N}_{\pi(x)}(\Pi_x(a_x))=0$. Then,
\[
\widetilde{\Pi}_x\overline{N}_x(\overline{\Pi}_x(a_x))=0.
\]
Since $\widetilde{\Pi}_x\colon A_x/\ker N^k_x\to
\tau_{\widetilde{A}}^{-1}(\pi(x))$ is an isomorphism, we deduce
that
\[
\overline{N}_x(\overline{\Pi}_x(a_x))=0
\]
or, equivalently $\overline{\Pi}_x N_x(a_x)=0,$ i.e.
\[
a_x\in \ker N^{k+1}_x=\ker N^k_x.
\]
It follows that
\[
\Pi_x(a_x)=\widetilde{\Pi}_x\overline{\Pi}_x(a_x)=0.
\]
In consequence, $\widetilde{N}_{\pi(x)}$ is injective and thus,  it is bijective.

Now we show that $\widetilde{P}$ is nondegenerate. Denote by $P$ the Poisson bisection associated with $\Omega.$ Let
$\widetilde{\al}_{\pi(x)}\in \widetilde{A}^{\ast}_{\pi(x)}$ be
such that
\[
0=\widetilde{P}_{\pi(x)}^\sharp(\widetilde{\al}_{\pi(x)})=\Pi_x
P_x^\sharp(\Pi_{x}^\ast \widetilde{\al}_{\pi(x)}).
\]
Using that $\widetilde{\Pi}_x\colon A_x/\ker N_x\to
\tau_{\widetilde{A}}^{-1}(\pi(x))$ is an isomorphism, we deduce
that
\[
\overline{\Pi}_x P_x^\sharp(\Pi_{x}^\ast
\widetilde{\al}_{\pi(x)})=0,
\]
i.e. $P_x^\sharp(\Pi_{x}^\ast \widetilde{\al}_{\pi(x)})\in
\ker N^k_x$. It follows that
\[
0= N^k_x P_x^\sharp(\Pi_{x}^\ast \widetilde{\al}_{\pi(x)})=
P_x^\sharp N^{\ast k}_x(\Pi_{x}^\ast
\widetilde{\al}_{\pi(x)}).
\]
Since $P_x$ is nondegenerate,
\[
N^{\ast k}_x(\Pi_{x}^\ast \widetilde{\al}_{\pi(x)})=0.
\]
Note that $N_x^{\ast k} (\Pi_{x}^\ast
\widetilde{\al}_{\pi(x)})= \widetilde{N}^{\ast k}_{\pi(x)}
(\widetilde{\al}_{\pi(x)})$ and that $\widetilde{N}$ is
nondegenerate. Hence, we deduce that $\widetilde{\al}_{\pi(x)}=0$.
\end{proof}
Under  the hypotheses of the previous theorem,  we will denote by
$\widetilde{\Omega}$ the symplectic section defined by
$\widetilde{\Omega}^\flat=-(\widetilde{P^\sharp})^{-1}$.

We summarize  the two steps of the reduction procedure given in
 Theorems
\ref{thm:SN:Lie:algebroid}  and \ref{thm:nondegenerate:N} in the following theorem.

\begin{thm}\label{thm:symplectic:Nijenhuis:nondegenerate:N}
Let $(A,\br_A,\rho_A, P,N)$ be a Poisson-Nijenhuis Lie algebroid
such that
\begin{enumerate}
\item[$i)$]
The Poisson structure $P$ has constant rank in the leaves of the
foliation $D=\rho_{A}(P^\sharp(A^\ast))$.
\end{enumerate}
\noindent If $L$ is a leaf of $D$, then, we have a
symplectic-Nijenhuis Lie algebroid structure
$(\br_{A_L},\rho_{A_L},\Omega_L, N_L)$ on
$A_L=P^\sharp(A^\ast)_{|L}\to L$.

Assume, moreover, that
\begin{enumerate}
\item[$ii)$] The induced Nijenhuis tensor $N_L:A_L\to A_L$ has constant Riesz index $k$;
\item[$iii)$] The dimension of the subspace $B_x=\ker N_x^k$ is constant, for all $x\in L$ (thus, $B=\ker N_L^k$ is a vector subbundle of $A$);
\item[$iii)$] The foliations $\rho_A(B)$ and $\mathcal{F}^B$ are regular, where
$$(\mathcal{F}^B)_a=\{X^c(a)+Y^v(a)/X,Y\in \Gamma(B)\},\mbox{ for }a\in A_L$$
\item[$iv)$] \emph{(condition $\mathcal{F}^{B})$}\; For all $x\in L$, $a_x-a'_x\in B_x$ if $a_x$ and $a'_x$ belong to the same leaf of the
foliation $\mathcal{F}^{B}$. %%%\emph{(condition \F)}\;
\end{enumerate}
Then, we obtain a symplectic-Nijenhuis Lie algebroid structure
\linebreak
$(\br_{\widetilde{A_L}},\rho_{\widetilde{A_L}},\widetilde{\Omega_L},\widetilde{N_L})$
on the vector bundle
$\widetilde{A_L}=A_L/\mathcal{F}^{B}\to\widetilde{L}=L/\rho_{A_L}(B)$ with $\widetilde{N_L}$ nondegenerate.
\end{thm}

We deal now with the particular case of a Poisson-Nijenhuis manifold $(M, \Lambda,N)$ such that $\Lambda^\sharp: A^*\to A$ has constant rank on each leaf $L$ of the characteristic foliation  $D=\Lambda^\sharp(T^*M)$ of $\Lambda.$ If we consider the Poisson-Nijenhuis Lie algebroid $(TM,\Lambda,N)$, then we can induce a symplectic-Nijenhuis structure $(P_L,N_L)$ on the leaf $L.$   If we suppose that $N_L$ has constant Riesz index $k$ and the foliation $\ker N_L^k$ is regular, then the leaf of the foliation ${\mathcal F}^{\ker N_L^k}$  is the tangent space of the leaf of $\ker N_L^k.$ Therefore, the condition ${\mathcal F}^{\ker N_L^k}$ holds. Consequently, from Theorem \ref{thm:symplectic:Nijenhuis:nondegenerate:N}, we have a symplectic-Nijenhuis structure with a nondegenerate Nijenhuis tensor on the reduced manifold $L/(\ker N_L^k).$ Thus, we recover the Magri-Morosi reduction process  \cite{MagriMorosi}.

%%%%%%%%%%%%%%%%%%%%%%%%%%%%%%%%%%%%%%%%%%%%%%%%%%%%%%%%%%%%%%%%%%%%%%%%%%%%%%%%%%%
%%%%%%%%%%%%%%%%%%%%%%%%%%%%%%%%%%%%%%%%%%%%%%%%%%%%%%%%%%%%%%%%%%%%%%%%%%%%%%%%%%%
\section{An explicit example of reduction of a Poisson-Nijenhuis Lie algebroid}
\subsection{A $G$-invariant Poisson-Nijenhuis structure on the cotangent bundle of a semidirect product of Lie groups }   %
\label{sec:example:reduction}%%%%%%%%%%%%%%%%%%%%%%%%%%%%%%%%%%%%%%%%%%%%%%%%%%%%%%
%%%%%%%%%%%%%%%%%%%%%%%%%%%%%%%%%%%%%%%%%%%%%%%%%%%%%%%%%%%%%%%%%%%%%%%%%%%%%%%%%%%
Let $H_1$ and $H_2$ be two Lie groups with Lie algebras $\h_1$ and
$\h_2$, respectively. Assume that there is an action $\phi\colon H_1
\times H_2\rightarrow H_2$  of $H_1$ on $H_2$ by Lie group isomorphisms and consider the semidirect product
$G=H_1 \times_\phi H_2$ whose operation is defined by
$$
(h_1,h_2)\cdot(h'_1,h'_2)=(h_1\cdot h'_1, h_2\cdot \phi(h_1,h_2')).
$$
Note that $H_2$ is a normal subgroup of $G$. The Lie algebra
associated to $G=H_1 \times_\phi H_2$ is $\g=\h_1 \times_\Phi h_2$
with the bracket
$$
[(\xi_1,\xi_2),(\eta_1,\eta_2)]_\g=\left([\xi_1,\eta_1]_{\h_1},
\Phi(\xi_1,\eta_2)- \Phi(\eta_1,\xi_2)+ [\xi_2,\eta_2]_{\h_2}\right),
$$
for all $\xi_1,\eta_1\in\h_1$, $\xi_2,\eta_2\in\h_2$, where
$\Phi=T_{(e_1,e_2)}\phi\colon \h_1 \times \h_2\rightarrow \h_2$ is the action induced
by $\phi\colon H_1 \times H_2\rightarrow H_2$. We remark  that $\h_1$ is
a Lie subalgebra and $\h_2$ is an ideal of $\g$. Consider now
$M=T^*G$. It may be identified with $G\times\g^*$ as follows:
\begin{align*}
M=T^*G &\longrightarrow G\times\g^*, \;\;\;\;\;\alpha_g \in T^*_gG \longmapsto (g, T^*_e l_g(\alpha_g))\in G\times {\frak g}^*,
\end{align*}
where $l_g\colon G\rightarrow G$ denotes the left translation by
$g\in G$. Under the identification $T^*G \cong G\times\g^*$, the
canonical symplectic structure of $T^*G$
\begin{align*}
\Omega \colon G\times\g^* \longrightarrow (\g\times\g^*)^* \times
(\g\times\g^*)^*
\end{align*}
is defined by
$$
\Omega_{(g,\mu)}\left((\xi,\pi),(\xi',\pi')\right)=\pi'(\xi)-\pi(\xi')+\mu([\xi,\xi']_\g),
$$
for all $\xi,\xi'\in\g$, $\pi,\pi'\in\g^*$.
%%% MEMO see Libermann-Marle book,  Lemma 4.2 page 205.
Note that $\Omega$ is $G$-invariant.

We define now on $T^*G$ a singular Poisson structure compatible with
$\Omega$. Let
$$
{\mathcal P}_{\g}\colon \g=\h_1 \times_\Phi \h_2 \longrightarrow \h_1
$$
be the canonical projection on the first factor. Then we have that  $\h_1\times
{\mathcal P}_{\g}^*(\h_1^*)\hookrightarrow\g \times\g^*$ is a symplectic subspace  of $T_{(e,\mu)}(
G\times\g^*)\cong \g \times\g^*$. Indeed,  let $\xi\in\h_1$, $\alpha\in\h_1^*$ be such
that
$$
\Omega_{(e,\mu)}\left((\xi,{\mathcal P}_{\g}^*(\alpha)),(\xi',{\mathcal P}_{\g}^*(\beta))\right)={\mathcal P}_{\g}^*(\beta)(\xi)-{\mathcal P}_{\g}^*(\alpha)(\xi')+\mu([\xi,\xi']_\g)=0,
$$
for all $\xi'\in\h_1$, $\beta\in\h_1^*$. Hence, we have
$$
\Omega_{(e,\mu)}\left((\xi,{\mathcal P}_{\g}^*(\alpha)),(0,{\mathcal P}_{\g}^*(\beta))\right)=\beta(\xi)=0
\Rightarrow \xi=0
$$
and
$$
\Omega_{(e,\mu)}\left((0,{\mathcal P}_{\g}^*(\alpha)),(\xi',{\mathcal P}_{\g}^*(\beta))\right)=-\alpha(\xi')=0
\Rightarrow \alpha=0.
$$

We consider now the symplectic subbundle
$$
\begin{array}{rcl}
\F_{\h_1}\colon (g,\mu)\in G\times\g^*\longmapsto
(T_el_g)(\h_1)\times {\mathcal P}_{\g}^*(\h_1^*)&=&T_{(e,\mu)}(l_g,id)(\h_1\times
{\mathcal P}_{\g}^*(\h_1^*))\\&&\subset T_{(g,\mu)}(G\times\g^*).
\end{array}
$$
We show that it is integrable. A basis of sections of this
subbundle is
$$
\{(\lxi,C_\alpha)\;|\;\xi\in\h_1,\,\alpha\in P^*(\h_1^*)\},
$$
where $\lxi$ is the left invariant vector field associated to $\xi$
and $C_\alpha$ is the vector field constant at $\alpha$. The bracket of these basic elements is given by
$$
[(\lxi,C_\alpha),(\lxi',C_\beta)]=([\lxi,\lxi'],0)=(\overleftarrow{[\xi,\xi']}_{\g},0).
$$
Since $\F_{\h_1}$ is symplectic, we have the decomposition
$$T(T^*G)\cong T(G\times {\frak g}^*)=\F_{\h_1}\oplus(\F_{\h_1})^\perp,$$
where $(\F_{\h_1})^\perp$ is the orthogonal to $\F_{\h_1}$ with
respect to the symplectic form $\Omega$.

Now, we define a Poisson bracket $\{\cdot,\cdot\}_{\h_1}$ in $T^*G$ as follows. For $f,g\in
C^\infty(T^*G)$,
$$
\{f,g\}_{\h_1}=\Omega(\Pa(\Ha_f^\Omega),\Pa(\Ha_g^\Omega))=\Pa(\Ha_{g}^\Omega(f))= \{f,g\}_\Omega-\Q(\Ha_g^\Omega)(f),
$$
where $\Pa\colon T(T^*G)\rightarrow\F_{\h_1}$ and $\Q\colon
T(T^*G)\rightarrow(\F_{\h_1})^\perp$ are the symplectic projectors and $\{\cdot,\cdot\}_\Omega$ is the Poisson bracket on $T^*G$ associated with the canonical symplectic structure of $T^*G$.
Here $\Ha_s^\Omega$ denotes the hamiltonian vector field of $s\in C^\infty(T^*G)$ with respect to the canonical symplectic structure of $T^*G$.

The symplectic foliation of $\{\cdot,\cdot\}_{\h_1}$ is $\F_{\h_1}$
since
$$
\Ha_g^{\{\cdot,\cdot\}_{\h_1}}=\Pa(\Ha_g^\Omega),
$$
where ${\mathcal H}_g^{\{\cdot,\cdot\}_{\h_1}}$ is the hamiltonian vector field of $g$ with respect to the Poisson bracket $\{\cdot,\cdot\}_{\h_1}.$

Keep into account that if $\theta\in T^*G$ and $L_\theta$ is the
leaf of $\F_{\h_1}$ passing through $\theta$, then we have
$$
\Ha_{f\smc\iota_\theta}^{\iota_\theta^*\Omega}=
\Pa(\Ha_f^\Omega)_{|L_\theta},
$$
where $\iota_\theta\colon L_\theta\hookrightarrow T^*G$ is the canonical inclusion. Note that
$$(\iota(\Pa(\Ha_f^\Omega)_{|L_\theta})\iota_\theta^*\Omega)(\Pa X)=\Omega_{|L_\theta}(\Ha_{f|L_\theta}^\Omega, \Pa X)=d(f\circ \iota_\theta)(\Pa X).$$ Thus,
$$
\{f\smc\iota_\theta,g\smc\iota_\theta\}_{\iota_\theta^*\Omega}=\left(\{f,g\}_{\h_1}\right)_{|L_\theta}\qquad\mbox{for
all }f,g\in C^\infty(T^*G).
$$
Therefore $L_\theta$ is the leaf of the symplectic foliation of
$\{\cdot,\cdot\}_{\h_1}$ through the point $\theta$.

It is clear that the Poisson bracket $\{\cdot,\cdot\}_{\h_1}$ is $G$-invariant.

We now study the compatibility between $\{\cdot,\cdot\}_{\h_1}$ and
$\{\cdot,\cdot\}_\Omega$. We know that
$$
\{f,g\}_{\h_1}= \{f,g\}_\Omega-\Q(\Ha_g^\Omega)(f).
$$
Next, we check that $\{f,g\}_{\h_2}= \Q(\Ha_g^\Omega)(f)$ is
Poisson.

Since $\F_{\h_1}$ and $\Omega$ are
$G$-invariant, then $(\F_{\h_1})^\perp$ is $G$-invariant.
Therefore
for describing $(\F_{\h_1})^\perp$ is enough to know $(\F_{\h_1})^\perp(e,\mu).$ A direct computation proves that
$$
(\F_{\h_1})^\perp(e,\mu)=\{(\xi,\pi)\in\g \times \g^*\;|\;
\xi\in\ker\Pa, \; \pi_{|\h_1}=-\xi_{\g^*}(\mu)_{|\h_1}\},
$$
where $\xi_{\g^*}(\mu)=-ad^*_\xi\mu$. Hence
\begin{align*}
(\F_{\h_1})^\perp(g,\mu)&=T_{(e,\mu)}(l_g,id) \left(
(\F_{\h_1})^\perp(e,\mu) \right)\\
&=\{(v_g,\pi)\in T_g G^* \times \g^*\;|\; (T_g
l_{g^{-1}})(v_g)\in\ker\Pa , \;
\pi_{|\h_1}=-\xi_{\g^*}(\mu)_{|\h_1}\}.
\end{align*}
The sections of $(\F_{\h_1})^\perp$ are of the form
\begin{align*}
\{(\overleftarrow{\xi},X) \;|\; \xi\in \ker\Pa,\; X\in \X(\g^*),\;
X(\mu)_{|\h_1}=-\xi_{\g^*}(\mu)_{|\h_1}, \forall \mu\in \g^*
\}\subseteq \X(G)\times \X(\g^*).
\end{align*}
and the brackets of them
$$
[(\overleftarrow{\xi},X),(\overleftarrow{\xi'},Y)]=([(\overleftarrow{\xi},\overleftarrow{\xi'}],[X,Y])
=(\overleftarrow{[\xi,\xi']_\g},[X,Y]).
$$
with $\xi,\xi'\in \ker\Pa$, $X,Y\in \X(\g^*)$ such that
$X(\mu)(\hat{\eta})=\mu([\xi,\eta])$ and $Y(\mu)(\hat{\eta})=\mu([\xi',\eta])$, for all
$\mu\in\g^*, \eta\in \h_1$. Here $\hat{\eta}:{\frak g}^*\to \Rr$ is the linear function induced by $\eta.$

Since $\xi,\xi'\in \h_2$ and $\h_2$ is a Lie subalgebra of $\g$, it
follows that $[\xi,\xi']_\g\in\h_2$. If $\mu\in\g^*$ and
$\eta\in\h_1$, then
$$
[X,Y](\mu)(\hat{\eta})=X(\mu)(Y(\hat{\eta}))-Y(\mu)(X(\hat{\eta})).
$$
Moreover,
$Y(\hat{v})(\mu')=Y(\mu')(\hat{\eta})=\mu'([\xi',\eta]_\g)$, for all
$\mu'\in\g^*$. Therefore, keeping in account that $\h_1$ is an ideal
in $\g$ we get
\begin{align*}
X(\mu)(Y(\hat{\eta}))&=X(\mu) (\widehat{[\xi',\eta]}_\g)= \mu ([\xi,[\xi',{\eta}]_\g]_\g), \\
Y(\mu)(X(\hat{\eta}))&=\mu ([\xi',[\xi,{\eta}]_\g]_\g).
\end{align*}
Hence
$$
[X,Y](\mu)(\hat{v})=\mu ([\xi,[\xi',{\eta}]_\g]_\g+
[\xi',[{\eta},\xi]_\g]_\g)=-\mu ([{\eta},[\xi,\xi']_\g]_\g) =
\mu([[\xi,\xi']_\g,\eta]_\g).
$$
Therefore $(\F_{\h_1})^\perp$ is a symplectic foliation, so that we
can consider the Poisson bracket $\{\cdot,\cdot\}_{\h_2}$ associated
to $(\F_{h_1})^\perp$, given by
$$
\{f,g\}_{\h_2}=\Q(\Ha_g^\Omega)(f).
$$
Thus $\{\cdot,\cdot\}_{\Omega}$ and $-\{\cdot,\cdot\}_{\h_1}$ are
compatible, since
$$\{\cdot,\cdot\}_{\Omega}+(-\{\cdot,\cdot\}_{\h_1})=\{\cdot,\cdot\}_{\h_2}.$$

Consequently, we can consider the Poisson-Nijenhuis manifold $(T^*G,\Omega,N)$, where $N=\Lambda_{\h_1}^\sharp\smc\Omega^\flat$ and $\Lambda_{\h_1}^\sharp:T^*(TG)\to T(T^*G)$ is the morphism induced by the Poisson bracket $\{\cdot, \cdot\}_{\h_1}.$  Using that $\{\cdot,\cdot\}_{\h_1}$ is $G$-invariant, it follows that $N$ is $G$-invariant.

\subsection{The Poisson-Nijenhuis Lie algebroid and its reduction}
We consider the action of $G$ on $T^*G\cong G\times {\frak g}^*$ by left translations, that is
\begin{align*}
G\times (G\times\g^*) &\longrightarrow G\times\g^*\\
(g',(g,\eta)) &\longmapsto (g'\cdot g,\eta).
\end{align*}
and let $\pi\colon T^*G \rightarrow T^*G/G$ be the corresponding
principal $G$-bundle.
 Since $\Omega$ and $N$ are $G$-invariant, we can consider
the corresponding Atiyah algebroid on
$$\tilde\pi\colon (T(T^*G))/G \longrightarrow T^*G/G.$$
We denote by $(\br,\rho)$ the Lie algebroid structure on
$\tilde\pi:(T(T^*G))/G)\to (T^*G)/G$.

Note that $\Gamma(\widetilde{\pi})$ may be identified with the set $\X^G(T^*G)$ of $G$-invariant vector fields on $T^*G$ and that if $X\in \X^G(T^*G)$ then $X$ is $\pi$-projectable. In fact, $\rho(X)=(T\pi)(X).$
Using Proposition  \ref {2.6}, we obtain a
Poisson-Nijenhuis structure on $\tilde\pi$ which we denote by
$(\tilde\Lambda,\tilde N)$. The foliation defined by the
distribution $D=\rho(\tilde\Lambda^\sharp((T^*(T^*G))/G))$ has just
one leaf which is the whole $(T^*G)/G$, since
$\Omega^\sharp((\Omega^1(T^*G))^G)=\X^G(T^*G)$ and these vector
fields generate all the vector fields in $T^*G$. In fact, $\tilde\Lambda$ is nondegenerate on $\tilde\pi: (T(T^*G))/G\to T^*G/G.$

Next, we compute $\ker N.$

Let $X\in\X^G(T^*G)$ be such that $N(X)=0$. Then, we have
$\Lambda_{\h_1}^\sharp\smc\Omega^\flat(X)=0$ and hence
$\Omega^\flat(X)\in \ker\Lambda_{\h_1}^\sharp$. Now,
$$
\alpha\in \ker\Lambda_{\h_1}^\sharp \Longleftrightarrow
\Lambda_{\h_1}^\sharp(\alpha)=\Pa(\Omega^\sharp(\alpha))=0
\Longleftrightarrow \alpha\in\Omega^\flat(\ker\Pa).
$$
Therefore,  $\ker \tilde N=(\F_{\h_1})^\perp$. Note that
$(\F_{\h_1})^\perp$ is a $G$-invariant foliation and hence it is
regular.

Let $X\in\X^G(T^*G)$ be such that $\tilde N^2(X)=0$. Then, we have
$\tilde N(X)\in \ker \tilde N=\ker\Pa$. That is,
$\Omega^\flat(X)\in(\Lambda_{\h_1}^\sharp)^{-1}(\ker\Pa)$. Now,
$$
\alpha\in(\Lambda_{\h_1}^\sharp)^{-1}(\ker\Pa) \Longleftrightarrow
\Lambda_{\h_1}^\sharp(\alpha)\in \ker \Pa \Longleftrightarrow
\Omega^\sharp(\alpha)\in \ker \Pa,
$$
since $\Omega^\sharp=\Lambda_{\h_1}^\sharp+\Lambda_{\h_2}^\sharp$
and $\Lambda_{\h_2}^\sharp(\alpha)\in \ker \Pa$. Hence $\ker \tilde
N^2=(\F_{\h_1})^\perp$. Therefore, the Riesz index is 1.

We study now the foliation ${\mathcal F}^{\ker N}$. The complete and vertical lifts of
the sections of $\ker N$ are complete and vertical lifts of
$G$-invariant vector fields in $T^*G$. Since $\ker N$ is regular,
then ${\mathcal F}^{\ker N}$ is regular.

 Then, if $L_\theta$ is the leaf of ${\ker N}$  passing
through $\theta$, we have that the leaf of ${\mathcal F}^{\ker N}$ passing through $v_\theta$ is
 $$
v_\theta+T L_\theta=v_\theta + (\bigcup_{x\in L_\theta}T_x L_\theta)=v_\theta + (\bigcup_{x\in L_\theta}\ker N(x)).
$$
Note that the condition ${\mathcal F}^{\ker N}$ is therefore also satisfied and hence
Theorem \ref{thm:symplectic:Nijenhuis:nondegenerate:N} can be
applied.

Finally, note that this example can be generalized if we consider a
Lie group $G$ with Lie algebra $\g$,  $\h$ a Lie subalgebra of
$\g$ and ${\mathcal P}_{\g}\colon \g\longrightarrow \h$ a projector
$({\mathcal P}_{\g |\h}=1_\h)$ such that $\ker{\mathcal P}_{\g}$ is an ideal of $\g$ and $\Pa$
is linear. Thus, on $T^*G$ we can define two compatible Poisson
structures (one of them being the canonical symplectic structure on
$T^*G$) and hence we can induce a Poisson-Nijenhuis structure on $T^*G$.

%%%%%%%%%%%%%%%%%%%%%%%%%%%%%%%%%%%%%%%%
%Acknowledgments
%%%%%%%%%%%%%%%%%%%%%%%%%%%%%%%%%%%%%%
\section{Acknowledgments}

This work has been partially supported by MICIN (Spain) grants
MTM2008-03606-E, MTM2009-13383, and by the Canary Government grants
SOLSUBC200801 000238 and ProID20100210. A.D.N. is also supported by
the FCT grant PTDC/MAT/\allowbreak 099880/2008 and thanks MICIN for
a Juan de La Cierva contract at University of La Laguna during which
part of this work was done.

%%%%%%%%%%%%%%%%%%%%%%%%%%%%%%%

\end{document}